\documentclass[11pt, reqno]{amsart}
\usepackage[utf8]{inputenc}
\usepackage{amsmath}
\usepackage{amsfonts}
\usepackage{amssymb}
\usepackage{amsthm}
\usepackage{csquotes}
\usepackage[english]{babel}
\usepackage[shortlabels]{enumitem}
\usepackage{lmodern}
\usepackage{cite}

\usepackage{color}

\newtheorem{theorem}{Theorem}[section]

\newtheorem{lemma}[theorem]{Lemma}
\newtheorem{proposition}[theorem]{Proposition}
\newtheorem*{theorem*}{Theorem~\ref{th:characterization}}

\theoremstyle{definition}

\theoremstyle{remark} \theoremstyle{remark}
\newtheorem{remark}[theorem]{Remark}
\newtheorem{example}[theorem]{Example}

\DeclareMathOperator{\id}{id}

\newcommand{\R}{{\mathbb R}}

\numberwithin{equation}{section}

\begin{document}

\title[]{Para-Sasakian $\phi-$symmetric spaces}

\author[E. Loiudice]{Eugenia Loiudice}
 \address{Philipps Universit\"at Marburg, Fachbereich Mathematik und Informatik, Hans-Meerwein-Straße, 35032 Marburg, Germany}
\email{loiudice@mathematik.uni-marburg.de}

\begin{abstract}
We study the Boothby--Wang fibration of para-Sasakian manifolds and introduce the class of para-Sasakian $\phi$-symmetric spaces, canonically fibering over para-Hermitian symmetric spaces. Using this fibration we give a method to explicitly construct semisimple para-Sasakian $\phi$-symmetric spaces. We provide moreover an example of non-semisimple para-Sasakian $\phi$-symmetric space. 
\end{abstract}



\maketitle

\section{Introduction}
Para-contact metric manifolds were introduced in \cite{KaneyukiWilliams} as odd-dimensional analogue  
of para-Hermitian spaces. Actually, they appear naturally when considering $S^1$-bundles over para-Hermitian spaces (see Section~2 of \cite{KaneyukiWilliams}). In particular, $S^1$-bundles over para-K\"ahler spaces are para-Sasakian manifolds. 
Examples of para-Sasakian manifolds are also obtained starting from some contact metric spaces: in \cite[Theorem~3.1]{mino-carriazo-molina} the authors showed that all contact metric $(\kappa,\mu)$-spaces with Boeckx invariant $I\in(-1,1)$ admit natural modifications (suggested by the bi-Legendrian structure on these manifolds, \cite{mino-diterlizzi}) of their structure tensors providing a para-Sasakian structure on the same manifold. 
Looking at the classification of contact metric $(\kappa,\mu)$-spaces \cite{loiudice-lotta}, we know that $(\kappa,\mu)$-spaces with Boeckx invariant $I\in(-1,1)$ locally fiber via the canonical Boothby-Wang fibration over the paracomplexification of a sphere, that is in particular a para-Hermitian symmetric space.

Motivated from this example we introduce in the present paper the class of para-Sasakian $\phi$-symmetric spaces. We show that, similarly to the Riemannian case, para-Sasakian $\phi$-symmetric spaces correspond (via the Boothby--Wang fibration) to para-Hermitian symmetric spaces, see Theorem~\ref{th:characterization}. In the same theorem we characterize the $\phi$-symmetric condition using the curvature tensor of both the Levi Civita connection and of the canonical para-contact connection.
Moreover, we show in Theorem~\ref{th:hom} that the closure $G \subset \text{Iso}(M)$ in the isometry group of a para-Sasakian manifold $M$ of the subgroup generated by all the $\phi$-geodesic symmetries acts transitively on $M$.

In section~\ref{sec:semisimple} we explicitly construct semisimple, para-Sasakian $\phi$-symmetric manifolds $G/H$ (namely para-Sasakian $\phi$-symmetric manifolds with $G$ semi- simple) starting from homogeneous semisimple, para-Hermitian symmetric spaces $G/K$. 
We remark that, in contrast to the Riemannian case, the center of the isotropy group $H$ of a simple para-Hermitian symmetric space $G/H$ is not always one-dimensional, indeed it can also have dimension $2$ (see \cite[Section 9]{koh}); moreover the metric of a simple para-Hermitian symmetric space $G/H$ is not necessarily the $G$-invariant extension of the Killing form of $G$ (see Proposition~\ref{prop:metric}).
Finally, in section~\ref{sec:non-semisimple} we consider a quadratic extension of the Lie algebra $\mathfrak{sl}(2,\mathbb{R})$ with a given involution, obtaining an example of non-semisimple para-Sasakian $\phi$-symmetric space.

The author thanks Ines Kath for suggesting the non-semisimple example and Antonio Lotta and Oliver Goertsches for useful discussions.

\section{Preliminaries}

A $2n$-dimensional smooth manifold $M$ endowed with 
an \emph{almost para-complex structure}, that is a $(1,1)$-tensor $I$ satisfying
\begin{itemize}
 \item $I^2=id$
 \item for every $p\in M$, the $\pm 1$-eigenspaces of $I_p$ are both $n$-dimensional subspaces of $T_pM$,
\end{itemize}
is called an \emph{almost para-complex} manifold.
If $I$ is in addition \emph{integrable}, i.e.,  if its Nijenhuis tensor 
 \[
  N_I(X,Y):=[IX,IY]-I[IX,Y]-I[X,IY]+[X,Y]
 \]
vanishes identically on $M$, then $I$ is a \emph{para-complex structure} on $M$ and $M$ is a \emph{para-complex manifold}.

Observe that the existence of an (integrable) almost para-complex structure on $M$ is equivalent to the existence of a pair of rank $n$ (integrable) distributions $(\mathcal{D}_+,\mathcal{D}_-)$, satisfying
\[
 TM=\mathcal{D}_+\oplus \mathcal{D}_-.
\]

\begin{example}
The product of any two equal-dimensional manifolds together with the obvious decomposition of its tangent bundle is a para-complex manifold.
 \end{example}
\begin{example}
In \cite[Proposition 6.3]{ALEKSEEVSKY-MEDORI-TOMASSINI} the authors showed that the invariant para-complex structures on a homogeneous manifold $G/H$ are in one-to-one correspondence with decompositions of the Lie algebra $\mathfrak{g}$ of $G$ of type
   $$
    \mathfrak{g}=\mathfrak{g}_+ + \mathfrak{g}_-,
   $$
with $\mathfrak{g}_+,\mathfrak{g}_-$ $Ad_H$-invariant subalgebras, such that $\mathfrak{g}_+\cap \mathfrak{g}_-=\mathfrak{h}$, $\dim \mathfrak{g}_+/\mathfrak{h}=\dim \mathfrak{g}_-/\mathfrak{h}$.
Explicit examples of homogeneous para-complex (in particular para-Hermitian symmetric) manifolds will appear later. 
 \end{example}

\begin{example}
Observe that any even dimensional sphere $S^{2n}$ does not admit an almost para-complex structure:
 Assume that the tangent bundle of $S^{2n}$ is the Whitney sum of two non-trivial subbundles
 \begin{equation}\label{es:TS}
   TS^{2n}=\mathcal{D}_+\oplus \mathcal{D}_-.
 \end{equation}
 As $S^{2n}$ is $(2n-1)$-connected, the bundles $\mathcal{D}_+$ and $\mathcal{D}_-$ have vanishing Euler class. Moreover, by \eqref{es:TS} the Euler class $e(TS^{2n})$ of the tangent bundle of $S^{2n}$ satisfies
 \[
  e(TS^{2n})=e(\mathcal{D}_+)\cup e(\mathcal{D}_-)
 \]
and hence it vanishes, which is a contradiction.
\end{example}

 An (almost) para-complex manifold $(N,I)$ endowed with a pseudo-Rie- mannian metric $g$ which is \emph{compatible} with $I$, namely such that
\[
 g(IX, Y)+g(X,IY)=0, 
\]
is called an \emph{(almost) para-Hermitian} manifold.
If moreover the \emph{fundamental $2$-form} $\omega$ of $(N,I,g)$
\[
 \omega(X,Y):=g(X,IY),
\]
is closed, then the (almost) para-Hermitian manifold $(N,I,g)$ is called \emph{(almost) para-K\"ahler}.

\vspace{0.2 cm}

Observe that in this case $\omega$ is symplectic and its restriction to each of the two $n$-dimensional distributions $\mathcal{D}_{-},\mathcal{D}_+$ determined by $I$ vanishes, namely $\mathcal{D}_{-},\mathcal{D}_+$ are integrable \emph{Lagrangian distributions}.

\vspace{0.2 cm}

Para-Hermitian and para-K\"ahler manifolds were defined by Libermann in \cite{Libermann1}, \cite{Libermann2}. 

\vspace{0.2 cm}

We recall that a symplectic manifold $(M,\omega)$ endowed with two complementary Lagrangian distributions (resp. Lagrangian foliations) $\mathcal{D}_1,\mathcal{D}_2$ is called an \emph{almost bi-Lagrangian} (resp. a \emph{bi-Lagrangian}) manifold, and $(\mathcal{D}_1,\mathcal{D}_2)$ is an \emph{almost bi-Lagrangian} (resp. a \emph{bi-Lagrangian}) structure on $M$.

\vspace{0.2 cm}
In \cite[Theorem 5, Theorem 6]{Etayo} the authors observed that (almost) bi-Lagrangian manifolds and (almost) para-K\"ahler manifolds are the same object, namely they proved the following proposition.
\begin{proposition}
 Let $M$ be a smooth manifold. There is a bijection between (almost) bi-Lagrangian structures on $M$ and (almost) para-K\"ahler structures on $M$.
\end{proposition}

One can observe that (cf. \cite[Theorem 1.3]{Merker}) the existence of an almost bi-Lagrangian structure on an (almost) symplectic manifold $(M,\omega)$ is equivalent to the reduction of the structure group of $TM$ from $Sp(2n)$ to the para-unitary group 
\[
 U(n,\mathbb{A}):=\left\{
A \in M_{n}(\mathbb{A}) \;|\; A \cdot \bar{A}=Id
\right\}.
\]
Here $\mathbb{A}$ denotes the algebra of para-complex numbers
\[
 \mathbb{A}:=\{a+kb\;| \; a,b\in \mathbb{R} \text{ and } k^2=1\}, 
\]
where the conjugation $\bar{\cdot}:\mathbb{A}\rightarrow \mathbb{A} $ is defined by
\[
 \overline{a+kb}=a-kb
\]

Combining this result with the fact that any symplectic manifold admits a compatible (tame) almost complex structure $J$ and a compatible almost para-complex structure $I$ satisfying 
\[
 IJ+JI=0
\]
(see \cite[Lemma 3.4]{Merker}), we have that the structure group of $TM$ is reducible from $Sp(2n)$ to 
$U(n,\mathbb{A})\cap U(n)=O(n)$.

\vspace{0.2 cm}
This gives a topological obstruction for the existence of a compatible almost para-complex structure on $(M,\omega)$ (cf. \cite[Proposition 4.1]{Merker}), as the reducibility of the structure group of $TM$ to $O(n)$ implies that the odd Chern classes of $(M,\omega)$ vanish.

\vspace{0.2 cm}

We recall that a para-Hermitian manifold $(N,I,g)$ admitting
at every point $p\in N$ a \emph{symmetry},
i.e., an automorphism 
$$s_p:N\rightarrow N,$$
of the para-Hermitian structure that reverses the geodesics trough $p$, namely such that
\[s_p(\gamma(t))=\gamma(-t),\] 
for every $\gamma$ geodesic, with $\gamma(0)=p$,
is called \emph{para-Hermitian symmetric} space.
It is known that (see \cite{KaneyukiKozai2}), every para-Hermitian symmetric space $(N,I,g)$ is para-K\"ahler, with authomorphism group $Aut(N,I,g)$ acting transitively on it.

\vspace{0.2 cm}

As in the next sections we will consider principal $G$-bundles, with $G=\mathbb{R},S^1$, over para-Hermitian symmetric spaces (which are in particular para-Sasakian manifolds), we recall the definition of contact and of para-Sasakian manifolds.

\vspace{0.2 cm}
A \emph{contact manifold} is a $(2n+1)$-dimensional manifold $M$ together with a $1$-form $\eta$, called \emph{contact form}, satisfying 
\[
 \eta\wedge (d\eta)^n\neq 0,
\]
everywhere. 
We have that $M$ is orientable and the subbundle $\ker \eta$ of the tangent bundle of $M$ determines a totally nonintegrable distribution $D$ of rank $2n$, called the \emph{horizontal distribution} of $(M,\eta)$. 
For every $p\in M$, the restriction of $d\eta$ to the fiber $D_p$ of $D$ 
\[
 d \eta_p|_{D_p\times D_p}:D_p\times D_p \rightarrow \mathbb{R},
\]
is nondegenerate, and 
\[
 T_pM= D_p \oplus \ker d\eta_p.
\]
The vector field $\xi$ on $M$ such that 
\[
 \xi_p\in \ker d\eta_p, \quad \eta_p(\xi_p)=1,
\]
for every $p\in M$, is called the \emph{Reeb vector field} of the contact manifold $(M,\eta)$, and any vector field $X$ such that $X_p\in D_p$ for every $p\in M$ is called \emph{horizontal}.

\vspace{0.2 cm}

A \emph{para-contact metric} manifold is a contact manifold $(M,\eta)$ together with a $(1,1)$-tensor $\phi$ and a semi-Riemannian metric $g$ satisfying
\begin{equation}\label{eq:paracontactmatric1}
 \phi^2=Id-\eta\otimes \xi, 
\end{equation}
\begin{equation}\label{eq:paracontactmatric2}
 g(\phi X,\phi Y)=-g(X,Y)+\eta(X)\eta(Y),
\end{equation}
\begin{equation}\label{eq:paracontactmatric3}
  d\eta(X,Y)= g(X,\phi Y).
\end{equation}
and such that $\phi$ is a para-complex structure when restricted to each fiber of the horizontal distribution.

\vspace{0.2 cm}
Let $(M,\phi,\xi,\eta,g)$ be a para-contact metric manifold.
Observe that the metric $g$ has signature $(n+1,n)$, and the tangent bundle $TM$ splits as follows
\[
 TM=\mathcal{D}_+\oplus \mathcal{D}_-\oplus \R\xi,
\]
where $\mathcal{D}_+, \mathcal{D}_-$ are the (equally dimensional) eigensubbundles of $\phi$ with eigenvalues $+1$ and $-1$ respectively.
If the \emph{Nijenhuis tensor} $N_{\phi}$ of $\phi$,
\[
 N_{\phi}(X,Y):=\phi^2[X,Y]+[\phi X,\phi Y]-\phi[\phi X,Y]-\phi[X,\phi Y],
\]
satisfies $$N_{\phi}(X,Y)-2d\eta(X,Y)\xi=0,$$
for every $X,Y\in TM$, then $M$ is called \emph{para-Sasakian} manifold.
In this case the Reeb vector field $\xi$ is Killing (namely $M$ is a \emph{$K$-para-contact} manifold)  and 
\begin{equation}\label{nabla xi}
 \nabla_X\xi=-\phi X,
\end{equation}
see \cite[Lemma~2.5]{zamkovoy}.
In \cite{zamkovoy}, the following useful properties on para-Sasakian manifolds are proved:
\begin{equation}\label{nablaphi}
 (\nabla_X \phi)Y=-g(X,Y)\xi+\eta(Y)X,
\end{equation}
(see \cite[Theorem 2.8]{zamkovoy}).
Using \eqref{nablaphi} one obtains (see \cite[Proposition 3.3]{zamkovoy}),
\begin{equation}\label{R}
 R(X,Y,\xi)=\eta(X)Y-\eta(Y)X,
\end{equation}
and hence
\begin{equation}\label{R2}
 R(X,\xi,Y)=g(X,Y)\xi-\eta(Y)X.
\end{equation}
Moreover, by \cite[Lemma 3.13]{zamkovoy}, we have that 
\begin{equation}\label{RR}
 \begin{aligned}
  R(\phi X,Y,Z)=&-R( X,\phi Y,Z)-g(\phi Z,X)Y-g(Z,\phi Y)X\\
   &+g(Z,X)\phi Y-g(Z,Y)\phi X.
 \end{aligned}
\end{equation}

\section{Para-Sasakian $\phi$-symmetric spaces}
After observing that on para-Sasakian manifolds the locally symmetric condition is too strong (see Propostion~\ref{prop:locally symmetric}), we consider on these manifolds a weaker condition, namely we define para-Sasakian $\phi$-symmetric spaces. We prove that para-Sasakian $\phi$-symmetric spaces are homogeneous and naturally fiber on para-Hermitian symmetric spaces.

\vspace{0.3 cm}

\begin{lemma}\label{lemma:nablaRxi}
 Let $(M, \phi,\xi,\eta,g)$ be a para-Sasakian manifold. Then, for all vector fields $X,Y,V$ on $M$

  \begin{equation}\label{eq:nablaRxi}
 \begin{aligned}
  (\nabla_V R)(X,Y,\xi) &=g(Y,\phi V)X-g(X,\phi V) Y+R(X,Y,\phi V) \\
                        &=g(Y,V)\phi X-g(X,V)\phi Y+\phi(R(X,Y,V)).\\
 \end{aligned}
 \end{equation}
\begin{equation}\label{eq:nablaRxi2}
 \begin{aligned}
  (\nabla_V R)(X,\xi,Y) &=g(Y,\phi X)V-g(V,Y)\phi X -R(\phi X,V,Y) \\
                        &=g(Y,\phi V)X-g(X,Y)\phi V +R(X,\phi V,Y).\\
 \end{aligned}
 \end{equation}
 \begin{equation}\label{eq:nablaRxi3}
 \begin{aligned}
  (\nabla_V R)(\xi,X,Y) &=g(Y, X)\phi V-g(\phi V,Y)X -R(\phi V,X,Y) \\
                        &=-g(Y,\phi X)V+g(V,Y)\phi X -R(V,\phi X,Y).\\
 \end{aligned}
 \end{equation}
\end{lemma}

\begin{proof}
Let $X,Y,V$ vector fields on $M$. Using \eqref{nabla xi}, \eqref{R}, \eqref{RR}, we have
 \begin{equation*}
 \begin{aligned}
  (\nabla_V R)(X,Y,\xi) =& \nabla_V(R(X,Y,\xi))-R(\nabla_VX,Y)\xi-R(X,\nabla_VY)\xi\\
                         &-R(X,Y,\nabla_V\xi)\\
                        =& \nabla_V(\eta(X)Y-\eta(Y)X)-\eta(\nabla_VX)Y+\eta(Y)\nabla_VX\\
                         &-\eta(X)\nabla_VY+\eta(\nabla_VY)X+R(X,Y,\phi V)\\
                        =& -g(X,\phi V) Y+g(Y,\phi V)X+R(X,Y,\phi V)\\
                        =& g(Y,V)\phi X-g(X,V)\phi Y+\phi(R(X,Y,V)),
 \end{aligned}
 \end{equation*}
 and hence \eqref{eq:nablaRxi}.
 Now, using \eqref{eq:nablaRxi} and basic properties of the curvature tensor, we obtain:
 \begin{equation*}
  \begin{aligned}
   g((\nabla_V R)&(X,\xi,Y),Z)= -g((\nabla_V R)(X,Z,\xi),Y)\\
   &=-g(\phi Y,X)g(V,Z)-g(V,Y)g(\phi X,Z)+g(R(Y,Z,V),\phi X)\\
   &=g(-R(\phi X,V,Y)-g(\phi Y,X)V-g(V,Y)\phi X,Z).
  \end{aligned}
 \end{equation*}
 Then, \[
       (\nabla_V R)(X,\xi,Y)=-R(\phi X,V,Y)-g(\phi Y,X)V-g(V,Y)\phi X,
      \]
      and the first equality in \eqref{eq:nablaRxi2} is proved. The second equality follows by \eqref{RR}.
      Since
      \[
       g((\nabla_VR)(\xi,X,Y),Z)=-g((\nabla_VR)(X,\xi,Y),Z),
      \]
\eqref{eq:nablaRxi3} follows from \eqref{eq:nablaRxi2}.
\end{proof}

Observe that, using \eqref{eq:nablaRxi} and \eqref{R}, we obtain:

\begin{proposition}\label{prop:locally symmetric}
Para-Sasakian, locally symmetric manifolds have constant curvature $-1$.
\end{proposition}

\vspace{0.2 cm}
Let $(M,\phi,\xi,\eta,g)$ be a para-Sasakian manifold.
A \emph{$\phi$-geodesic symmetry} at $x\in M$ is a local diffeomorphism $\sigma$ at $x$ such that for every \emph{$\phi$-geodesic} $\gamma$ (that is a geodesic satisfying $\eta(\gamma'(t))=0$ for every $t$), with $\gamma(0)$ in the trajectory of the integral curve of $\xi$ passing through $x$, we have that
\begin{equation*}
 \sigma(\gamma(s))=\gamma(-s),
\end{equation*}
for every $s$ such that $\gamma(s)$ is in the domain of $\sigma$.

 We say that a para-Sasakian manifold $M$ is a \emph{locally $\phi$-symmetric} space if for every point $x\in M$ there exists a local automorphism of the contact metric structure which is also a $\phi$-geodesic symmetry at $x$.
 
 If any $\phi$-geodesic symmetry of a para-Sasakian manifold $M$ is extendable to a global automorphism of the para-contact metric structure and $\xi$ generates a $1$-parameter group of global transformations, then $M$ is called \emph{$\phi$-symmetric} or also \emph{globally $\phi$-symmetric} space.

 \vspace{0,2 cm}

 Observe that for every horizontal vector $v\in T_pM$ there is a $\phi$-geodesic with initial condition $(p,v)$ (see equation \eqref{eq:existence geodesic}), and hence we have that 
 a local automorphism $\sigma_p$ at $p\in M$ is a $\phi$-geodesic symmetry at $p$ if and only if
 \[ \sigma_p(p)=p, \quad (d \sigma_p)|_{\mathcal{D}_p}=-\id_{\mathcal{D}_p}.\]
Thus, two $\phi$-geodesic symmetries at $p$ which are also automorphisms of the contact metric structure, coincide on a neighborhood of $p$.

 \begin{theorem}\label{th:hom}
 Let $(M,\phi,\xi,\eta,g)$ be a connected para-Sasakian $\phi$-symmetric space, and $Iso(M)$ be the Lie group of isometries of $M$. Then,
the closure $G$ in $Iso(M)$ of the subgroup generated by all the $\phi$-geodesic symmetries of $M$, acts transitively on $M$. 
 \end{theorem}
 
 \begin{proof}
Let $p\in  M$ and $v\in \mathcal{D}_p\subset T_pM$ be any horizontal vector. We denote by $\gamma$ the geodesic with initial condition $(p,v)$. As $\xi$ is Killing, 
\begin{equation}\label{eq:existence geodesic}
 \frac{d}{dt} g_{\gamma(t)}(\gamma'(t),\xi_{\gamma(t)})
 =g(\gamma'(t), \nabla_{\gamma'(t)}\xi)=0,
\end{equation}
and thus, since $g_{\gamma(0)}(\gamma'(0),\xi_{\gamma(0)})=0$,  $\gamma'(t)$ is orthogonal to $\xi$ for every $t$.
Observe that the image of the geodesic $\gamma$ is contained in the orbit $G\cdot p$:
let $\sigma$ be the $\phi$-symmetry at $p$. We fix any element $\gamma(s)$ in the image of $\gamma$, and consider the $\phi$-geodesic 
$$\tilde{\gamma}:t\mapsto \gamma(t+s/2).$$
and the $\phi$-symmetry $\tilde{\sigma}$ at $\gamma(s/2)$. We have that for every $t$
\[
 \tilde{\sigma}(\tilde{\gamma}(t))=\tilde{\gamma}(-t).
\]
In particular 
\[
 \tilde{\sigma}(\tilde{\gamma}(s/2))=\tilde{\gamma}(-s/2)=\gamma(0), 
\]
and hence $\tilde{\sigma}({\gamma}(s))=\gamma(0)=p$ and 
\[
 \gamma(s)\in G \cdot p.
\]
This implies that $\mathcal{D}_p \subset T_p (G\cdot p)$. As the contact distribution $\mathcal{D}$ is $G$-invariant, $\mathcal{D}_q \subset T_q (G\cdot p)$ for every $q\in G \cdot p$. Moreover $\mathcal{D}$ is completely nonintegrable and of codimension $1$, then $T_q (G\cdot p)=T_qM$ and as $M$ is connected, $G \cdot p =M$.
 \end{proof}

 \begin{remark} 
  Note that the proof of Theorem~\ref{th:hom} holds also in the Riemannian setting, providing an alternative and less technical proof of the homogeneity of Sasakian $\phi$-symmetric spaces, cf. \cite[Theorem~7]{BLAIR-VANHECKE}. 
 \end{remark}
 
 \vspace{0.2 cm}
 
 In the next theorem we consider the Boothby--Wang fibration $$\pi:M\rightarrow M/\xi$$ of \emph{$K$-para-contact} manifolds (namely para-contact metric manifolds whose Reeb vector field is Killing) and of para-Sasakian manifolds, finding some useful relations between the curvature tensor of the total and of the base space of the fibration. Then, in Theorem~\ref{th:characterization}, we characterize para-Sasakian locally $\phi$-symmetric spaces using
 \begin{enumerate}
  \item The curvature tensor field of the Levi-Civita connection $\nabla$ and of the \emph{canonical para-contact connection} $\widetilde{\nabla}$, defined in \cite[Section 4]{zamkovoy} by
\[\widetilde{\nabla}_XY:=\nabla_XY+\eta(X)\phi(Y)+\eta(Y)\phi(X)+d\eta(X,Y)\xi, \]
with $X,Y$ vector fields on $M$.
 \item The Boothby--Wang fibration: we will see that para-Sasakian $\phi$-symmetric spaces correspond to \emph{para-Hermitian symmetric} spaces. 
 \end{enumerate}

  \vspace{0.3 cm}
 We first recall the correspondence between \emph{regular contact} and symplectic manifolds given by the Boothby--Wang theorem (for details see \cite{BW}, or also \cite[Section 7.2]{book:geiges}):

   \vspace{0.1 cm}

A contact manifold $(M^{2n+1},\eta)$ is called \emph{regular} if the associated Reeb vector field $\xi$ is \emph{regular},
namely if for every point $p\in M$ there is an open neighborhood of $p$ that is pierced at most once by any integral curve of $\xi$.

\vspace{0.2 cm}
Note that, if $(M,\eta)$ is a connected regular contact manifold, the orbit space $M/\xi$ is a smooth manifold (see for example \cite[Theorem~VIII]{palais}). 
We have moreover that, if $M$ is contact homogeneous, then it is automatically regular (see \cite[Theorem~4]{BW}).

\vspace{0.2 cm}

The Boothby--Wang theorem says that the contact form $\eta$ of a connected, either homogeneous or compact regular contact manifold $(M,\eta)$ is, up to rescaling, the connection form of a principal $G$-bundle
\[
 \pi:M\rightarrow M/\xi,
\]
with $G=S^1$ or $G=\R$ (according to the circumstance that the integral curves of $\xi$ are compact or not)
over the symplectic manifold $(M/\xi,\omega)$ where $\omega$ is the curvature form of $\eta$.
Conversely, let $(B,\omega)$ be a symplectic manifold with $[\omega]$ an integral cohomology class, then the associated principal $G$-bundle 
\[
 \pi:M\rightarrow B
\]
with Euler class $[\omega]$, has a connection $1$-form $\eta$ which is a regular contact form on $M$, with curvature form $\omega$.

\vspace{0.3 cm}
 The proof of the next two theorems in this sections works similarly as in the Sasakian setting (cf. \cite{ogiue}, \cite{takahashi}). Moreover a part of the proof of Theorem~\ref{th:R-R*} is contained in \cite[Section~2]{KaneyukiWilliams}. We include the proof here for completeness.

\begin{theorem}\label{th:R-R*}
 Let $(M, \phi,\xi,\eta,g)$ be either a compact regular, or a homogeneous, $K$-para-contact manifold. Then, 
 the tensors $\phi$, $g$ induce on the base space $M/\xi$ of the Boothby--Wang fibration $\pi:M\rightarrow  M/\xi$, an almost para-K\"ahler structure $(I,\bar{g})$ given by
 \begin{equation}\label{Induced para-complex}
  I_p(X)=d \pi_q(\phi X^*), \quad \bar{g}_p(X,Y)=g_q(X^*,Y^*),
 \end{equation}
for all vector fields $X,Y$ on $M/\xi$, $p\in M/\xi$ and $q\in M$ such that $\pi(q)=p$, and where $X^*,Y^*$ denote the lift of $X$ and $Y$ with respect to the connection form $\eta$. 
Let $\nabla$, $\bar{\nabla}$, and $R$, $\bar{R}$, denote respectively the Levi-Civita connections and the curvature tensors of $g$ and $\bar{g}$, then
\begin{equation}\label{nabla}
\begin{aligned}
 (\bar{\nabla}_XY)^* = \nabla_{X^*}Y^*-\frac{1}{2}\eta([X^*,Y^*])\xi= \phi^2(\nabla_{X^*}Y^*),
\end{aligned}
\end{equation}
\begin{equation}\label{curvature}
\begin{aligned}
 (\bar{R}(X,Y,Z))^*=&\phi^2\left( R(X^*,Y^*,Z^*)+g(Z^*,\phi Y^*)\phi{X^*}\right.\\
 &\left.-g(Z^*,\phi X^*)\phi{Y^*} - 2g(\phi X^*, Y^*)\phi{Z^*}\right),
\end{aligned}
\end{equation}
for all vector fields $X,Y,Z$ on $M/\xi$.

In particular if $(M, \phi,\xi,\eta,g)$ is para-Sasakian, then $(M/\xi,I,\bar{g})$ is para-K\"ahler, and for all vector fields $X,Y,Z,W$ on $M/\xi$
\begin{equation}\label{curvature paraS}
 \begin{aligned}
  (\bar{R}(X,Y,Z))^*=&  R(X^*,Y^*,Z^*)+g(Z^*,\phi Y^*)\phi{X^*}-g(Z^*,\phi X^*)\phi{Y^*} \\
                     &-2g(\phi X^*, Y^*)\phi{Z^*},
 \end{aligned}
\end{equation}
\begin{equation}\label{phi2nablaR}
 \left((\bar{\nabla}_X\bar{R})(Y,Z,W)\right)^*=\phi^2\left((\nabla_{X^*}{R})(Y^*,Z^*,W^*)\right).
\end{equation}
\end{theorem}

\begin{proof}
Observe that, since $\phi$ and $g$ are invariant under the flow of $\xi$, then $I$ and $\bar{g}$ are well defined. Moreover for every vector field $X$ on $M/\xi$,
 \begin{equation*}
  I^2(X)=I(d\pi(\phi X^*))=d \pi (\phi^2 X^*)=d\pi (X^*)=X,
 \end{equation*}
and hence, as the eigensubbundles $\mathcal{D}^+$ and $\mathcal{D}^-$ corresponding to the eigenvalues $1$ and $-1$ of $\phi$, have both dimension $n$, $I$ is an almost paracomplex structure on $M/\xi$. For all $X,Y$ vector fields  on $M/\xi$
\begin{equation*}
\begin{aligned}
  \bar{g}(IX,IY)&=\bar{g}(d\pi(\phi X^*), d\pi(\phi Y^*))=g(\phi X^*, \phi Y^*)=-g(X^*,Y^*)\\
                &=-\bar{g}(X,Y),
\end{aligned}
\end{equation*}
and thus $(M/\xi, I,\bar{g})$ is an almost para-Hermitian manifold. As the fundamental $2$-form $\omega$ of $M$ is closed and it is the pullback of 
\[                                                                                             
\Omega(X,Y):=\bar{g}(X,IY)                                                                                                                    \]
 via $\pi$, we have that $d \Omega =0$ and hence $(M/\xi, I,\bar{g})$ is an almost para-K\"ahler manifold.
We prove now \eqref{nabla}. Let $X,Y,Z$ be vector fields on $M/\xi$. Since $\xi$ is Killing,
\begin{equation*}
\begin{aligned}
 2g(\nabla_{X^*}Y^*,\xi)=&X^*(g(Y^*,\xi))+Y^*(g(X^*,\xi))-\xi(g(X^*,Y^*))\\
                         &-g(X^*,[Y^*,\xi])+g(Y^*,[X^*,\xi],)+g(\xi,[X^*,Y^*])\\
                        =& g(\xi,[X^*,Y^*])=\eta([X^*,Y^*]),
\end{aligned}
\end{equation*}
and hence 
\[
 \phi^2(\nabla_{X^*}Y^*)=\nabla_{X^*}Y^*-\frac{1}{2}\eta[X^*,Y^*]\xi.
\]
We also have that
 \begin{equation*}
\begin{aligned}
2\bar{g}(d \pi(\nabla_{X^*}Y^*),Z)=& 2g(\nabla_{X^*}Y^*,Z^*)\\
                           =&X^*(g(Y^*,Z^*))+Y^*(g(X^*,Z^*))-Z^*(g(X^*,Y^*))\\
                            &-g(X^*,[Y^*,Z^*])+g(Y^*,[X^*,Z^*],)+g(Z^*,[X^*,Y^*])\\
                           =&2\bar{g}(\bar{\nabla}_XY,Z).
\end{aligned}
\end{equation*}
Then $d \pi(\nabla_{X^*}Y^*)=\bar{\nabla}_XY$ and 
$$(\bar{\nabla}_XY)^*=\nabla_{X^*}Y^*-\frac{1}{2}\eta([X^*,Y^*])\xi=\phi^2(\nabla_{X^*}Y^*).$$
Using \eqref{nabla xi} and the fact that $\xi$ commutes with the horizontal lift $W^*$ of any vector field $W$ on $M$ (as $W^*$ is invariant under the action of the structure group of $\pi:M\rightarrow M/\xi$), we have  
\begin{equation*}
 \begin{aligned}
  (\bar{\nabla}_{[X,Y]}Z)^*=&\phi^2({\nabla}_{[X,Y]^*}Z^*)\\
                           =&\phi^2\left( {\nabla}_{[X^*,Y^*]}Z^* -\eta[X^*,Y^*] \nabla_{\xi}Z^*\right)\\
                           =&\phi^2\left( {\nabla}_{[X^*,Y^*]}Z^* -2g(\nabla_{X^*}Y^*,\xi) (\nabla_Z^*{\xi}+[\xi,Z^*])\right)\\
                           =&\phi^2\left( {\nabla}_{[X^*,Y^*]}Z^* -2(g(Y^*,\nabla_{X^*}\xi)-X^*(g(Y^*,\xi))) \phi Z^*)\right)\\
                           =& \phi^2\left( {\nabla}_{[X^*,Y^*]}Z^* +2g(Y^*,\phi X^*)\phi Z^* \right),
 \end{aligned}
\end{equation*}
\begin{equation*}
 \begin{aligned}
  (\bar{\nabla}_X\bar{\nabla}_YZ)^*=&\phi^2({\nabla}_{X^*}(\bar{\nabla}_YZ)^*)\\
                                   =&\phi^2\left({\nabla}_{X^*}{\nabla}_{Y^*}Z^*-\frac{1}  {2}X^*(\eta[Y^*,Z^*])\xi-\frac{1}{2}\eta[Y^*,Z^*]{\nabla}_{X^*}\xi\right)\\
                                   =& \phi^2\left({\nabla}_{X^*}{\nabla}_{Y^*}Z^*+g(Z^*,\phi Y^*)\phi X^*\right),
 \end{aligned}
\end{equation*}
and analogously \[
              (\bar{\nabla}_Y\bar{\nabla}_XZ)^*=\phi^2\left({\nabla}_{Y^*}{\nabla}_{X^*}Z^*+g(Z^*,\phi X^*)\phi Y^*\right).
             \]
Thus, equation \eqref{curvature} follows.

 Now we assume that $M$ is para-Sasakian. By \eqref{RR} we have:  
 \begin{equation*}
 \eta(R(X^*,Y^*,Z^*))=g(R(X^*,Y^*,\phi^2 Z^*),\xi)=-g(R(X^*,Y^*,\phi Z^*),\phi\xi)=0.
 \end{equation*}
 Then,
$$ \phi^2(R(X^*,Y^*,Z^*))=R(X^*,Y^*,Z^*)-\eta(R(X^*,Y^*,Z^*))\xi=R(X^*,Y^*,Z^*),$$
and hence using equation \eqref{curvature} we obtain \eqref{curvature paraS}.
 To prove that $I$ is integrable, we consider any $X,Y$ vector fields on $M/\xi$. Then
\begin{equation*}
\begin{aligned}
 N_I(X,Y)=&[IX,IY]+I^2[X,Y]-I[IX,Y]-I[X,IY]\\
         =&d \pi\left( [\phi X^*,\phi Y^*] + \phi^2[X,Y]-\phi[d\pi(\phi X^*),Y]^*-\phi[d\pi(\phi X^*),Y]^* \right)\\
         =& d \pi(N_{\phi}(X^*,Y^*))=0.
\end{aligned}
\end{equation*}
Finally, a straightforward computation using \eqref{nablaphi}, \eqref{nabla}, \eqref{curvature}, gives \eqref{phi2nablaR}.
\end{proof}

\begin{theorem}\label{th:characterization}
 Let $(M,\phi,\xi,\eta,g)$ be a para-Sasakian manifold. Let $R,\tilde{R}$ denote respectively the curvature tensors of the Levi-Civita connection $\nabla$ and of the canonical para-contact connection $\widetilde{\nabla}$.
 The following conditions are equivalent:
 \begin{enumerate}[(i)]
 \item $M$ is locally $\phi$-symmetric.
 \item For every $x\in M$, consider an open neighborhood $U$ of $x$ such that $(U,\eta|_U)$ is regular. Then the para-Sasakian structure on $U$ obtained by restriction from the given structure on $M$, induces on the base space of the local fibration  
 $$\pi:U\rightarrow U/\xi,$$
a para-Hermitian symmetric structure.
 \item For all horizontal vector fields $X,Y,Z,W$,
 $$\phi^2((\nabla_XR)(Y,Z,W))=0.$$ 
  \item  For all vector fields $X,Y,Z,W$,
  \begin{equation*}
          \begin{aligned}
           (\nabla_{X}R)(Y,Z,W)=&\left(g(X,Y)g(\phi Z,W)-g(X,Z)g(\phi Y,W)\right.\\
  &\left.-g(\phi R(Y,Z,X),W)\right)\xi \\
  &+ \eta(Y)\left( -g(W,\phi X)Z+g(W,Z)\phi X -R(Z,\phi X,W) \right)\\
  &+ \eta(Z)\left( g(W,\phi X)Y-g(W,Y)\phi X +R(Y,\phi X,W) \right)\\
  &+ \eta(W)\left( g(Z,\phi X)Y-g(Y,\phi X) Z+R(Y,Z,\phi X) \right).
          \end{aligned}
         \end{equation*}
  \item $\widetilde{\nabla}\widetilde{R}=0$.
 \end{enumerate}
\end{theorem}

 \begin{proof}
We show first that $(i) \Rightarrow (ii)$.
 Let $x$ be any point in $M$ and $U$ an open neighborhood of $x$ such that $(U,\eta|_{U})$ is regular. Let $\pi(p)$ be any element in $U/\xi$ and $\sigma$ be a $\phi$-symmetry at $p\in U$. Observe that, if $\pi(p)=\pi(q)$ then $\pi(\sigma(p))=\pi(\sigma(q))$: consider an integral curve 
$\alpha:[0,1]\rightarrow M$ of $\xi$ such that $\alpha(0)=p$ and $\alpha(1)=q$. Since $\sigma$ is an automorphism of the para-contact structure, it preserves $\xi$ and we have:
\[
(d\pi)_{\sigma(\alpha(t))}\circ (d\sigma)_{\alpha(t)}(\alpha'(t))=(d\pi)_{\sigma(\alpha(t))}\xi_{\sigma(\alpha(t))}=0.
\]
Hence $\pi(\sigma(p))=\pi(\sigma(q))$ and $\sigma$ induces a map $\widetilde{\sigma}:U/\xi\rightarrow U/\xi$,
\[
 \widetilde{\sigma}(\pi(p)):=\pi(\sigma(p)),
\]
that is a symmetry at $\pi(x)$ on $U/\xi$.
 
 The implication $(ii)\Rightarrow (iii)$ follows directly by equation~\eqref{phi2nablaR}. To prove $(iii) \Rightarrow (iv)$, we first consider $X,Y,Z,W$ horizontal vector fields.
 Then, by Lemma~\ref{lemma:nablaRxi} and basic properties of the curvature tensor, we have that
   \begin{equation}\label{eq:iv-horizontal}
    \begin{aligned}
     0=&\phi^2((\nabla_{X}R)(Y,Z,W))\\
     =&(\nabla_{X}R)(Y,Z,W)-\eta((\nabla_{X}R)(Y,Z,W))\xi\\
     =&(\nabla_{X}R)(Y,Z,W)-g(-(\nabla_XR)(Y,Z,\xi),W)\xi\\
     =& (\nabla_{X}R)(Y,Z,W)-\left(g(X,Y)g(\phi Z,W)-g(X,Z)g(\phi Y,W)\right.\\
  &\left.-g(\phi R(Y,Z,X),W)\right)\xi.
    \end{aligned}
   \end{equation}
 and hence $(iv)$ holds in this case. 
Now we consider $X,Y,Z,W$ any vector fields. Using \eqref{eq:iv-horizontal}, we have
 \begin{equation*}
  \begin{aligned}
   &(\nabla_{\phi^2 X}  R)(\phi^2 Y,\phi^2 Z, \phi^2 W)=\left(g(\phi^2 X,\phi^2 Y)g(\phi^3 Z, \phi^2 W)\right.\\
   &\left.-g(\phi^2 X,\phi^2 Z)g(\phi^3 Y,\phi^2 W)-g(\phi R(\phi^2 Y,\phi^2 Z,\phi^2 X),\phi^2 W)\right)\xi\\
   =&( g(\phi Z, W)(g(X,Y)-\eta(X)\eta(Y))-g(\phi Y, W)(g(X,Z)g(\eta(X)\eta(Z))\\
   &+g(R(\phi^2 Y,\phi^2 Z, \phi^2 X),\phi W))\xi,
  \end{aligned}
 \end{equation*}
and since
\begin{equation*}
 \begin{aligned}
  R&(\phi^2 Y,\phi^2 Z, \phi^2 X)=R(Y,Z,X)-\eta(X)R(Y,Z,\xi)-\eta(Z)R(Y,\xi,X)\\
  &+\eta(X)\eta(Z)R(Y,\xi,\xi)-\eta(Y)R(\xi,Z,X)+\eta(X)\eta(Y)R(\xi,Z,\xi)\\
  =&R(Y,Z,X)-\eta(X)(\eta(Y)Z-\eta(Z)Y)+\eta(Z)(g(Y,X)\xi-\eta(Z)Y)\\
  &+\eta(Z)\eta(X)\eta(Y)\xi-\eta(Z)\eta(X)Y+\eta(Y)(g(Z,X)\xi-\eta(X)Z)\\
  &-\eta(Z)\eta(X)\eta(Y)\xi+\eta(Y)\eta(X)Z\\
  =&R(Y,Z,X)+\eta(Z)\eta(X)Y-\eta(Y)\eta(X)Z\\
  &+\big(-\eta(Z)g(X,Y)+\eta(Y)g(Z,X)\big)\xi
 \end{aligned}
\end{equation*}
where for the second equality we used \eqref{R}, \eqref{R2}, we obtain
\begin{equation}\label{eq:rphi2}
 \begin{aligned}
  (&\nabla_{\phi^2 X}  R)(\phi^2 Y,\phi^2 Z, \phi^2 W)= \Big(g(\phi Z,W)g(X,Y)\\
  &-g(\phi Y, W)g(X,Z)+g(R(Y,Z,X),\phi W)\Big)\xi.
 \end{aligned}
\end{equation}
 On the other hand, as Lemma~5.3 of \cite{takahashi} holds also for para-Sasakian manifolds, we have that
 \[
   (\nabla_{\xi}R)(X_1,X_2,X_3)=0,
 \]
for all horizontal vector fields $X_1,X_2,X_3$. Then, using equations \eqref{eq:paracontactmatric1}, \eqref{R}, \eqref{R2} and Lemma~\ref{lemma:nablaRxi}, we have
\begin{equation*}
 \begin{aligned}
  (\nabla_{\phi^2 X}&  R)(\phi^2 Y,\phi^2 Z, \phi^2 W)=(\nabla_{X}  R)(\phi^2 Y,\phi^2 Z, \phi^2 W)\\
  =\,&(\nabla_{X}  R)(Y, Z, W)-\eta(W)(\nabla_{X}  R)(Y, Z, \xi)\\
  &-\eta(Z)(\nabla_{X}  R)(Y,\xi,W)+ \eta(Z)\eta(W)(\nabla_{X}  R)(Y, \xi, \xi)\\
  &-\eta(Y)(\nabla_{X}  R)(\xi,Z,W)+ \eta(Y)\eta(W)(\nabla_{X}  R)( \xi,Z, \xi)\\
  &+ \eta(Y)\eta(Z)(\nabla_{X}  R)( \xi,\xi,W)-\eta(Y)\eta(Z)\eta(W)(\nabla_{X}  R)( \xi,\xi,\xi)\\
  =\,&(\nabla_{X}  R)(Y, Z, W)\\
  &-\eta(W)\big(g(Z,X)Y-g(Y,\phi X)Z+R(Y,Z,\phi X)\big)\\
  &-\eta(Z)\big(g(W,\phi X)Y-g(Y,W)\phi X+R(Y,\phi X, W)\big)\\
  &-\eta(Y)\big(-g(W,\phi X)Z+g(Z,W)\phi X+R(\phi X,Z, W)\big).
 \end{aligned}
\end{equation*}
Finally, we compare this equation with equation \eqref{eq:rphi2} obtaining $(iv)$.
 
Now we prove that $(iv) \Rightarrow (v)$. We denote by $A$ the $(2,1)$ tensor 
\begin{equation*}
 \begin{aligned}
  A(X,Y):=&\widetilde{\nabla}_XY-\nabla_XY\\
 =&\eta(X)\phi(Y)+\eta(Y)\phi(X)+d\eta(X,Y)\xi.
 \end{aligned}
\end{equation*}
By \cite[Proposition~4.2, Theorem~4.10]{zamkovoy}, the structure tensors $(\phi,\xi,\eta,g)$ are $\tilde{\nabla}$-parallel and hence so is also $A$.
Then,
\begin{equation*}
 \begin{aligned}
  0=&(\tilde{\nabla}_VA)(X,Y)=\tilde{\nabla}_V(A(X,Y))-A(\tilde{\nabla}_VX,Y)-A(X,\tilde{\nabla}_VY)\\
  =& {\nabla}_V(A(X,Y))+A(V,(A(X,Y))-A({\nabla}_VX,Y)\\
   &-A(A(V,X),Y)-A(X,{\nabla}_VY)-A(X,A(V,Y)),
 \end{aligned}
\end{equation*}
and hence
\begin{equation*}
 \begin{aligned}
  \tilde{R}(X,Y,Z)=& \tilde{\nabla}_X\tilde{\nabla}_YZ-\tilde{\nabla}_Y\tilde{\nabla}_XZ-\tilde{\nabla}_{[X,Y]}Z \\
                  =&R(X,Y,Z)+A(X,{\nabla}_YZ)-\nabla_Y(A(X,Z))+A(\nabla_YX,Z)\\
                  &+A(X,A(Y,Z))-A(Y,A(X,Z))-A(Y,\nabla_XZ)\\
                  &+\nabla_X(A(Y,Z))-A(\nabla_XY,Z)\\
                  =&R(X,Y,Z)-A(X,A(Y,Z))+A(Y,A(X,Z))\\
                  &+A(A(X,Y),Z)-A(A(Y,X),Z)
 \end{aligned}
\end{equation*}
and 
\begin{equation*}
 \begin{aligned}
  (\tilde{\nabla}_V\tilde{R})(X,Y,Z)=& \tilde{\nabla}_V(\tilde{R}(X,Y,Z))- \tilde{R}(\tilde{\nabla}_VX,Y,Z)-\tilde{R}(X,\tilde{\nabla}_VY,Z)\\ 
  &- \tilde{R}(X,Y,\tilde{\nabla}_VZ)\\
  =&{\nabla}_V(\tilde{R}(X,Y,Z))+A(V,R(X,Y,Z))-R(\nabla_VX,Y,Z)\\
  &-R(A(V,X),Y,Z)-R(X,A(V,Y),Z)\\
  &-R(X,\nabla_VY,Z)-R(X,Y,\nabla_VZ)-R(X,Y,A(A(V,Z))\\
  &+ (\tilde{\nabla}_VA)(Y,A(X,Z))-(\tilde{\nabla}_VA)(X,A(Y,Z))\\
  & +(\tilde{\nabla}_VA)(A(X,Y),Z)-(\tilde{\nabla}_VA)(A(Y,X),Z)\\
  =& ({\nabla}_V {R})(X,Y,Z)+A(V,R(X,Y,Z))-R(A(V,X),Y,Z)\\
  &-R(X,A(V,Y),Z)-R(X,Y,A(V,Z)).
 \end{aligned}
\end{equation*}
By \eqref{R}, \eqref{RR} and the fact that 
\[
 A(X,Y)=\eta(X)\phi(Y)+\eta(Y)\phi(X)+d\eta(X,Y)\xi,
\]
we have
\begin{equation*}
 \begin{aligned}
 A(V,R(X,Y,Z))=&g(V,\phi R(X,Y,Z))\xi+\eta(V)\phi R(X,Y,Z)\\
 &+g(R(X,Y,Z),\xi)\phi V\\
 =&g\Big(V,R(X,Y,\phi Z)+g(Y,Z)\phi X-g(X,\phi Z)Y\\
 &+g(Y,\phi Z)X+g(X,Z)\phi Y+g(X,Z)\phi Y\Big)\xi\\
  &+\eta(V)\phi R(X,Y,Z)+g(\eta(Y)X-\eta(X)Y,Z)\phi V,\\
 R(A(V,X),Y,Z)=&g(V,\phi X)R(\xi,Y,Z)+\eta(V)R(\phi X,Y,Z)\\
 &+\eta(X)R(\phi V,Y,Z),\\
 R(X,A(V,Y),Z)=&g(V,\phi Y)R(X,\xi,Z)+\eta(V)R(X,\phi Y,Z)\\
 &+\eta(Y)R(X,\phi V,Z),\\
 R(X,Y,A(V,Z))=&g(V,\phi Z)R(X,Y,\xi)+\eta(V)R(X,Y,\phi Z)\\
 &+\eta(Z)R(X,Y,\phi V),\\
 \end{aligned}
\end{equation*}
and hence, using \eqref{R}, \eqref{R2} and \eqref{RR}, we obtain,
\[
 (\widetilde{\nabla}_V\widetilde{R})(X,Y,Z)=0.
\]
It remains to prove that $(v)\Rightarrow (i)$. Let $p\in M$ and consider the linear isomorphism $\alpha$ of $T_pM$, definite by 
\[
 \alpha(\xi_p)=\xi_p, \quad \alpha(e_i)=-e_i,
\]
where $\{e_1,\dots,e_{2n}\}$ is an orthonormal $\phi-$basis of $D_p$ (namely $\phi(e_i)=e_{n+i}$ for every $i\in\{1,\dots,n\}$).
Observe that 
\begin{equation*}
 \begin{aligned}
  \tilde{T}(X,Y)&=\tilde{\nabla}_XY-\tilde{\nabla}_YX-[X,Y]\\
  &= A(X,Y)-A(Y,X)
 \end{aligned}
\end{equation*}
and hence, as $A$ is $\tilde{\nabla}$-parallel (we know that the structure tensors of $M$ are $\tilde{\nabla}$-parallel by \cite[Proposition~4.2, Theorem~4.10]{zamkovoy}), we have that $\tilde{T}$ also is $\tilde{\nabla}$-parallel; we observe moreover that $\alpha$ preserves $\tilde{T}$ and $\tilde{R}$. Then, by Theorem~7.4 Chapter~6 of \cite{KN1}, there exists a local affine diffeomorphism $\sigma$ such that $d\sigma_p=\alpha$.
Since $\sigma$ preserves $\eta$ and $\phi$ (which are $\tilde{\nabla}$-parallel) at $p$ and commutes with the parallel transport (as it is an affine map), we have that it preserves $\eta$ and $\phi$ on a neighborhood of $p$, and then also $g$ and $\xi$ in the same neighborhood, giving a $\phi$-geodesic symmetry at $p$. 
\end{proof}
 
 \vspace{0.2 cm}
By the theorem above, we know that the curvature and torsion of the canonical para-contact connection $\tilde{\nabla}$ of a para-Sasakian locally $\phi$-symmetric space $(M,\phi,\xi,\eta,g)$ are $\tilde{\nabla}$-parallel. Moreover, by \cite[Proposition~4.2, Theorem~4.10]{zamkovoy} the structure tensors $\phi,\xi,\eta,g$ are also $\tilde{\nabla}$-parallel. Then Kiri\v{c}enko's theorem \cite{kiricenko} (see also \cite[Section 2.2, 2.3]{book:Calvaruso-Lopez}) applies, giving an alternative proof of the homogeneity of para-Sasakian locally $\phi$-symmetric spaces (cf. Theorem~\ref{th:hom}):
 
 \begin{theorem}\label{th:globally symmetric}
  Let $(M,\phi,\xi,\eta,g)$ be a simply connected, complete, para-Sasakian locally $\phi$-symmetric space. Then $M$ is globally $\phi$-symmetric and a homogeneous para-contact manifold.
 \end{theorem}

\begin{example}
The homogeneous space 
$SO(n+1,1)/SO(n)$ 
is a para-Sasakian space: by \cite[Theorem~17]{loiudice-lotta} it admits a contact metric $(\kappa,\mu)$ structure $(\varphi,\xi,\eta,g)$, and hence by \cite[Theorem~3.1]{mino-carriazo-molina} it admits a para-Sasakian structure, compatible with the given contact form $\eta$. 
We observe that the base space of the Boothby--Wang fibration
$$SO(n+1,1)/SO(n) \rightarrow SO(n+1,1)/(SO(n)\times SO(1,1))$$
is para-Hermitian symmetric (cf. \cite{KaneyukiKozai2}, Section~4), and hence by Theorem~\ref{th:characterization},
$$SO(n+1,1)/SO(n)$$
is a para-Sasakian $\phi$-symmetric space.
\end{example}

\section{Semisimple para-Sasakian $\phi$-symmetric spaces}\label{sec:semisimple}
In this section we construct semisimple, para-Sasakian $\phi$-symmetric manifolds fibering over semisimple para-Hermitian symmetric spaces.

\vspace{0.3 cm}
Let $(N,I,g)$ be a para-Hermitian symmetric space, and $G$ the connected component of the identity in the authomorphism group $Aut(N,I,g)$. We have that $G$ acts transitively on $N$ and hence $N$ is diffeomorphic to the coset space $G/H$, where $H\subset G$ denotes the isotropy subgroup at a point $o\in N$.
Let $\mathfrak{g}$, $\mathfrak{h}$ be the Lie algebras respectively of $G$ and $H$, and $\sigma$ be the derivative of the involution of $G$ --that we denote again with $\sigma$-- defined by
\[
 \sigma: G\rightarrow G; \quad a\mapsto s_o \circ a \circ s_o,
\]
with $s_o\in Aut(N)$ the symmetry at $o$. We have that the involutive automorphism $\sigma$ of $\mathfrak{g}$ fixes $\mathfrak{h}$, and the $\sigma$-eigenspace decomposition of $\mathfrak{g}$ is
\[
 \mathfrak{g}= \mathfrak{h}\oplus \mathfrak{n},
\]
where $\mathfrak{n}\cong T_oN$.
Observe that the involutive automorphism $\sigma$ of $\mathfrak{g}$,
the linear authomorphism $I_o$ and the nondegenerate bilinear symmetric form $\langle\;,\;\rangle$ induced respectively by $I$ and $g$, satisfy
\begin{enumerate}[(I)]
 \item $I^2=id$,\label{I}\label{pagref}
 \item $I$ commutes with $ad_W|_\mathfrak{n}$ for every $W\in  \mathfrak{h}$,\label{II}
 \item $\langle IX,Y\rangle + \langle X,IY\rangle=0$, for every $X,Y\in \mathfrak{n}$,\label{III}
 \item $\langle ad_X Y_1,Y_2\rangle + \langle Y_1, ad_X Y_2\rangle=0$, for every $X\in \mathfrak{h}$, $Y_1,Y_2\in \mathfrak{n}$;\label{IV}
\end{enumerate}
namely $(\mathfrak{g},\mathfrak{h},\sigma, I_o,\langle\;,\;\rangle)$ is a \emph{para-Hermitian symmetric system}.

On the other hand, if $(\mathfrak{g},\mathfrak{h},\sigma, I_o,\langle\;,\;\rangle)$ is a para-Hermitian symmetric system, we can consider the simply connected Lie group $G$ with Lie algebra $\mathfrak{g}$, and the connected Lie subgroup $H$ with $Lie(H)=\mathfrak{h}$. We have that $G/H$ together with the $G$-invariant tensors $I$ and $g$, naturally determined respectively from $I_o$ and $\langle\;,\;\rangle$, is a para-Hermitian symmetric space whose associated para-Hermitian symmetric system is $(\mathfrak{g},\mathfrak{h},\sigma, I_o,\langle\;,\;\rangle)$.
For more details on para-Hermitian symmetric spaces we refer the reader to \cite{KaneyukiKozai2}.

\vspace{0.3 cm}
A para-Hermitian symmetric space $(G/H,I,g)$ is called \emph{simple} (or respectively \emph{semisimple}) if the Lie group $G$ is simple (respectively semisimple).

\vspace{0.1 cm}
We remark that, in the Riemannian setting the isotropy group $H$ of a simple Hermitian symmetric spaces $G/H$ has one-dimensional center, and the orthogonal complement in $\mathfrak{h}=Lie(H)$ of a generator of the center $\mathfrak{z(h)}$ of $\mathfrak{h}$ with respect to the Killing form $B$ of $G$, is the commutator of $\mathfrak{h}$. Hence the Lie subgroup $K\subset H$ associated to the Lie algebra $[\mathfrak{h},\mathfrak{h}]$ is closed in $G$, and the Hermitian structure on $G/H$ induces naturally, via the submersion $\pi:G/K\rightarrow G/H$, a Sasakian $\phi$-symmetric structure on the manifold $G/K$ (see \cite{KowalskiJimenez}). 
In contrast, on a simple para-Hermitian symmetric space, the center of the isotropy group $H$ can be either one- or two-dimensional, corresponding to the simple group $G$ not admitting or admitting a complex structure (see \cite[Section 4]{koh}). 
We recall moreover that, as $\mathfrak{g}$ is simple, $(\mathfrak{g},\mathfrak{h})$ is effective, and then by \cite[Lemma~3.1]{KaneyukiKozai2}, there exists a unique element $Z\in \mathfrak{g}$ such that 
\begin{enumerate}[(a)]
 \item $ad_Z=I$;\label{Z1}
 \item The centralizer $c_{\mathfrak{g}}(Z)$ of $Z$ in $\mathfrak{g}$ is equal to $\mathfrak{h}$.\label{Z2}
 \end{enumerate}
By \cite[Lemma~3.1]{KaneyukiKozai1} we have that
the center $\mathfrak{z(h)}$ of $\mathfrak{h}$ is given by
 \begin{equation}\label{eq:center}
  \mathfrak{z(h)}=\begin{cases}
                  span_\mathbb{R}\{Z, iZ\}, \; \text{if } \mathfrak{g} \text{ admits a complex structure } \\
                  span_\mathbb{R}\{Z\}, \; \text{if } \mathfrak{g} \text{ does not admit a complex structure}, \\
                 \end{cases}
 \end{equation}
where $i$ is a complex structure on $\mathfrak{g}$ (commuting with the involution $\sigma$) defined in 
\cite[Theorem~6]{koh}.

\vspace{0.3 cm}

In the next proposition we extend Proposition~2 of \cite{shimokawa-sugimoto} --proved when the center of $H$ is one-dimensional-- to any simple, para-Hermitian symmetric space. For completeness we include also the proof in case $\dim (\mathfrak{z(h)})=1$. In particular, if $\dim (\mathfrak{z(h)})=2$ we obtain that the metric $g$ of a para-Hermitian symmetric space $G/H$ is not necessarily the $G$-invariant extension of a multiple of the Killing form $B$ of $\mathfrak{g}$.
In Theorem \ref{th:lambda} we find a codimension-one closed subgroup $K$ of $H$ such that the manifold $G/K$ is naturally endowed with a para-Sasakian $\phi$-symmetric structure, induced from the para-Hermitian one.

\begin{proposition}\label{prop:metric}
 Let $(G/H, I,g)$ be a simple, para-Hermitian symmetric space. We have:
 \begin{enumerate}
  \item If the center of $H$ is one-dimensional, then $g$ is the $G$-invariant extension of the Killing form of $G$ up to a constant.
  \item \label{2} If the center of $H$ is two-dimensional, then $\mathfrak{g}$ admits a complex structure $i$ such that the metric $g$ is the $G$-invariant extension of a bilinear, symmetric form of type
  \begin{equation}\label{eq:f lambda mu}
   f_{(\lambda,\mu)}:=\lambda B(\cdot,\cdot)+\mu B(i \cdot,\cdot), \quad \lambda,\mu\in \mathbb{R},
  \end{equation}
where $B$ denotes the Killing form of $\mathfrak{g}$, and $\lambda,\mu\in \mathbb{R}$, $(\lambda,\mu)\neq (0,0)$. 
Moreover, for every $\lambda,\mu\in \mathbb{R}$, $(\lambda,\mu)\neq (0,0)$ we have that $f_{(\lambda,\mu)}$ is an $ad_{\mathfrak{h}}$-invariant, nondegenerate symmetric bilinear form on the $-1$-eigenspace $\mathfrak{n}\cong T_oG/H$ of the involution $\sigma$ of $G/H$, with signature $(n,n)$, $2n=dim \,G/H$. We obtain hence a $2$-parameter family of pseudo-Riemannian metrics $g_{(\lambda,\mu)}$ such that $(G/H,I,g_{(\lambda,\mu)})$ is para-Hermitian symmetric space.
 \end{enumerate}

\end{proposition}
\begin{proof}
Let $\Omega$ denote the pullback via the natural projection
\[
 \pi:G\rightarrow G/H
\]
of the fundamental $2$-form $\omega$ of $G/H$. As para-Hermitian symmetric spaces are para-K\"ahler (see Theorem~\ref{th:R-R*}), $\omega$ is closed and then also the left invariant, $2$-form $\Omega$ on $G$ is so. We know moreover that, the first and second cohomology classes of $G$ vanish (since by assumption $G$ is simple), and thus there exists a unique $1$-form $\alpha$ such that 
\[
 d\alpha =\Omega. 
\]
Now, using the isomorphism
\[
 \mathfrak{g}\rightarrow \mathfrak{g}^*; \quad X \mapsto B(X,\cdot),
\]
we consider the unique element $A\in \mathfrak{g}$ satisfying
\begin{equation}\label{A}
 \alpha = B(A,\cdot);
\end{equation}
here as usual $\mathfrak{g}$ denotes the Lie algebra of $G$ and $\mathfrak{g}^*$ the dual space of $\mathfrak{g}$. 
Observe that, for every $X,Y\in \mathfrak{g}$
\begin{equation}\label{eq:B-omega}
 \begin{aligned}
 - B([A,X],Y)&= - B(A,[X,Y])=-\alpha([X,Y])=d\alpha(X,Y)\\
  &=\Omega(X,Y)=\omega(\pi_*X,\pi_*Y)\circ \pi,
 \end{aligned}
\end{equation}
and since $B$ and $\omega$ are both nondegenerate, this implies that
\begin{equation}\label{eq:CA}
 c_\mathfrak{g}(A)=\mathfrak{h}.
\end{equation}
Using the same notation as above, we consider the unique element $Z\in \mathfrak{g}$ such that 
\[
 ad_Z=I, \; c_{\mathfrak{g}}(Z)=\mathfrak{h},
\]
and by \eqref{eq:center}, we have
\begin{equation*}
 A=\begin{cases}
    \lambda Z, \; \text{if } \mathfrak{z(h)} \text{ is one-dimensional}\\
    \lambda Z+\mu \,i Z,  \; \text{if } \mathfrak{z(h)} \text{ is two-dimensional}
   \end{cases}
\end{equation*}
with $\lambda,\mu \in \mathbb{R}$. Since $ad_Z^2|_{\mathfrak{n}}=Id$, and
  \[
   g_o(X,Y)=\omega(X,IY)=\Omega(X,IY)=-B(A,[X,ad_ZY]),
  \]
for every $X,Y\in T_o(G/H)$, where $o=eH\in G/H$ and $e$ is the neutral element in $G$, we obtain:
\begin{itemize}
 \item If $\mathfrak{z(h)}$ is one-dimensional,
 \[
  g_o(X,Y)=-\lambda B(Z,[X,ad_ZY])=\lambda B(ad_Z^2 Y,X)= \lambda B(X,Y).
 \]
\item If $\mathfrak{z(h)}$ is two-dimensional,
\begin{equation*}
 \begin{aligned}
  g_o(X,Y)&=-\lambda B(Z,[X,ad_ZY])-\mu B(iZ,[X,ad_ZY])\\
  &= \lambda B(X,Y)+\mu B(X,ad_{iZ}ad_ZY)\\
  &=\lambda B(X,Y)+\mu B(X,iY).
 \end{aligned}
\end{equation*}
\end{itemize}
\vspace{0.2 cm}
 Now we prove the second part of \eqref{2}. 
We consider the the bilinear symmetric map 
\[
f_{(a,b)}(\cdot\,,\cdot):=a B(\cdot\,,\cdot)+b B(i \cdot,\cdot),
\]
with $a,b\in \mathbb{R}$, $(a,b)\neq(0,0)$. Observe that 
for every $X,Y\in \mathfrak{n}$
\begin{equation*}
 \begin{aligned}
 f_{(a,b)}(X,Y)&=a B(X,Y)+b B(i X,Y)\\
 &=2a \mathfrak{Re}(\tilde{B}(X,Y))+2b \mathfrak{Re}(\tilde{B}(i X,Y))\\
 &=2a \mathfrak{Re}(\tilde{B}(X,Y))-2b \mathfrak{Im}(\tilde{B}(X,Y)),
 \end{aligned}
\end{equation*}
where $\tilde{B}$ denote the Killing form of the complex $n$-dimensional Lie algebra $\mathfrak{g}$.
Since $\tilde{B}|_{\mathfrak{n}}:\mathfrak{n}\times \mathfrak{n}\rightarrow \mathbb{C}$ is non-degenerate ($B$ is by assumption non-degenerate on $\mathfrak{g}=\mathfrak{h}\oplus\mathfrak{n}$, and $B(\mathfrak{h},\mathfrak{n})=0$; then the restriction of $B$ to both $\mathfrak{h}$ and $\mathfrak{n}$ is non-degenerate and hence so is also the restriction of $\tilde{B}$ to $\mathfrak{h}$ and to $\mathfrak{n}$), bilinear symmetric form, there exists a basis $\{e_j\}_{j=1,\dots, n}$ of $\mathfrak{n}$ such that $\tilde{B}$ is represented by the identity matrix, and hence for every $X,U\in \mathfrak{n}$,
\[
 X=\sum_{j=1}^n (x_j+iy_j)e_j\equiv (x_1,y_1,\dots,x_n,y_n),
\]
\[
 U=\sum_{j=1}^n (u_j+iv_j)e_j\equiv (u_1,v_1,\dots,u_n,v_n),
\]
we have
\[
 \tilde{B}(X,U)=\sum_{j=1}^n x_ju_j-\sum_{j=1}^n y_jv_j + i\left( \sum_{j=1}^n x_jv_j+\sum_{j=1}^n y_ju_j\right).
\]
Then, the real and imaginary part of $ \tilde{B}$ are represented by the matrices 

\begin{equation*}
 \mathfrak{Re}\tilde{B}\equiv\begin{pmatrix}
  \begin{array}{cccccc}
\begin{pmatrix}
1 & 0\\
0 & -1
\end{pmatrix} & &  \\
 & \ddots\\
 &  & \begin{pmatrix}
1 & 0\\
0 & -1
\end{pmatrix}
\end{array}
 \end{pmatrix},
\end{equation*}

\begin{equation*}
 \mathfrak{Im}\tilde{B}\equiv\begin{pmatrix}
  \begin{array}{cccccc}
\begin{pmatrix}
0 & 1\\
1 & 0
\end{pmatrix} & &  \\
 & \ddots\\
 &  & \begin{pmatrix}
0 & 1\\
1 & 0
\end{pmatrix}
\end{array}
 \end{pmatrix},
\end{equation*}
and hence
\begin{equation*}
 f_{(a,b)}\equiv 2\cdot\begin{pmatrix}
  \begin{array}{cccccc}
\begin{pmatrix}
a & -b\\
-b & -a
\end{pmatrix} & &  \\
 & \ddots\\
 &  & \begin{pmatrix}
a & -b\\
-b & -a
\end{pmatrix}
\end{array}
 \end{pmatrix}
\end{equation*}
is a non-degenerate bilinear form with signature $(n,n)$. 
Since $Z$ commutes with every element in $\mathfrak{h}$ and $I=ad_Z$ commutes with $i$, it is easy to check that conditions \ref{I}, \ref{II}, \ref{III}, \ref{IV} are satisfied 
and hence $(G/H,I,g_{(a,b)})$ is a para-Hermitian symmetric space, where $g_{(a,b)}$ denotes the $G$-invariant metric determined by the $ad_{H}-$invariant bilinear symmetric form $f_{(a,b)}$. 
\end{proof}

\begin{theorem}\label{th:lambda}
Let $(G/H,I,\langle\,,\,\rangle)$ be a simple, para-Hermitian symmetric space. 
Then $H$ has one- or two-dimensional center and admits a codimension one closed subgroup $K$ such that $G/K$ is a para-Sasakian $\varphi$-symmetric manifold canonically fibering, via the Boothby--Wang fibration, over $G/H$. 
\end{theorem}
\begin{proof}
Let $(\mathfrak{g}, \mathfrak{h}, \sigma, I, \langle\,,\rangle)$ be the para-Hermitian symmetric system associated to $(G/H,I,\langle\,,\,\rangle)$, and 
\[
 \mathfrak{g}=\mathfrak{h}+\mathfrak{n},
\]
be the $\sigma$-eigenspace decomposition of $\mathfrak{g}$. We know that there exists a unique $Z\in \mathfrak{g}$ such that $I=ad_Z$ and the centralizer of $Z$ in $\mathfrak{g}$ equals $\mathfrak{h}$; moreover, the center of $\mathfrak{h}$ is one- or two-dimensional, when respectively $\mathfrak{g}$ does not admit or admits a complex structure. More precisely, by \eqref{eq:center} we have
\begin{equation*}
 \mathfrak{z(h)}=\begin{cases}
                  span_{\mathbb{R}}\{Z\}, \; \text{ if } \dim \mathfrak{z(h)}=1\\
                  span_{\mathbb{R}}\{Z,iZ\}, \; \text{ if } \dim \mathfrak{z(h)}=2
                 \end{cases}.
\end{equation*}
As in the proof of proposition~\ref{prop:metric}, we consider the unique element $A\in \mathfrak{g}$ determined by the fundamental $2$-form of $G/H$,
\begin{equation}\label{eq:A}
 A=\begin{cases}
    \lambda Z, \; \text{ if } \dim \mathfrak{z(h)}=1\\
    \lambda Z+\mu iZ, \; \text{ if } \dim \mathfrak{z(h)}=2
   \end{cases}
\end{equation}
with $\lambda,\mu\in \mathbb{R}$, see \eqref{A}.  By Proposition~\ref{prop:metric} we also know that the metric $g$ at the point $o=eH\in G/H$, with $e\in G$ identity element
$$g_o:\mathfrak{n}\times \mathfrak{n}\rightarrow \mathbb{R},$$
is of the form
\[
 g_o=\begin{cases}
            \lambda B, \text{ if } \dim \mathfrak{z(h)}=1\\
            \lambda B(\cdot,\cdot)+\mu B(i \cdot,\cdot), \text{ if } \dim \mathfrak{z(h)}=2
           \end{cases}.
\]
As $\mathfrak{h}$ coincides with the centralizer of a semisimple element,
it is reductive and decomposes as
\begin{equation}\label{eq:reductive}
 \mathfrak{h}=[\mathfrak{h},\mathfrak{h}]\oplus\mathfrak{z(h)}.
\end{equation}
 We consider the subalgebra $\mathfrak{k}\subset \mathfrak{h}$ given by
\begin{equation}\label{eq:k}
 \mathfrak{k}:=\begin{cases}
                [\mathfrak{h},\mathfrak{h}], \text{ if } \dim \mathfrak{z(h)}=1\\
               [\mathfrak{h},\mathfrak{h}]+span_{\mathbb{R}}\{C\}, \text{ if } \dim\mathfrak{z(h)}=2
              \end{cases},
\end{equation}
where $C$ is any element in $\mathfrak{z(h)}$, and show that the associated subgroup is closed. 

As $G$ is by assumption simple, the symmetric pair $(G,H)$ is effective; hence the isotropy representation is faithful, and the subgroup corresponding to $[\mathfrak{h},\mathfrak{h}]$ is closed (see for instance \cite[Corollary 16.2.8]{Hilgert-Neeb}). 

Thus, in case $\mathfrak{z(h)}$ is $1$-dimensional, the subgroup $K$ associated to $\mathfrak{k}=[\mathfrak{h},\mathfrak{h}]$ is closed.
In case $\mathfrak{z(h)}$ is $2$-dimensional,
considering a Cartan involution of $\mathfrak{g}$ which commutes with $\sigma$
\[
 \mathfrak{g}=\mathfrak{r}+i\mathfrak{r},
\]
where $\mathfrak{r}$ is a maximal compact Lie subalgebra of $\mathfrak{g}$, we have that 
$$\mathfrak{z(h)}= \mathfrak{z(h)}\cap  \mathfrak{r}+ \mathfrak{z(h)}\cap i \mathfrak{r}$$
where $\mathfrak{z(h)}\cap  \mathfrak{r}$, $\mathfrak{z(h)}\cap i \mathfrak{r}$ are both $1$-dimensional real subspaces by \cite[Lemma 3.1]{KaneyukiKozai1} generated by $iZ$ and $Z$ respectively. Observe that the Lie group $Z(H)$ 
associated to $\mathfrak{z(h)}$ can not be isomorphic to $S^1\times S^1$ or to $\mathbb{R}^2$ and hence it is isomorphic to $S^1\times \mathbb{R}$. Since any Lie subgroup of $S^1\times \mathbb{R}$ is closed, the subgroup $S$ corresponding to $span_{\mathbb{R}}\{C\}\subset \mathfrak{z(h)}$ is closed in $Z(H)$ and hence also in $G$. We observe moreover that, as 
$$
\big[span_\mathbb{R}\{C\},[\mathfrak{h}, \mathfrak{h}]\big]\subset [\mathfrak{h}, \mathfrak{h}],
$$ 
the subgroup corresponding to $[\mathfrak{h}, \mathfrak{h}]+span_\mathbb{R}\{C\}$ is $\widetilde{H}S$, where $\widetilde{H}$ denotes the subgroup associated to $[\mathfrak{h},\mathfrak{h}]$ (see for instance \cite[Corollary 11.1.20]{Hilgert-Neeb}). Let $p$ denote the submersion
\[
 p:\widetilde{H}\times Z(H) \rightarrow (\widetilde{H}\times Z(H))/(\widetilde{H}\cap Z(H))=H.
\]
Observe that 
\[
 p^{-1}(p(\widetilde{H}\times S))
\]
is closed because $\widetilde{H}\times S$ is closed in $\widetilde{H}\times Z(H)$ and $\widetilde{H}\cap Z(H)$ is a finite cyclic group (see \cite[Lemma 2.2]{KaneyukiKozai1}). Then, $p(\widetilde{H}\times S)=\widetilde{H}S$ is closed in $H$ and hence also in $G$.

Then, in both cases the subgroup $K$ associated to $\mathfrak{k}$ is a closed subgroup of $G$, and hence $G/K$ is a manifold. Now we define on $G/K$ the following tensors:

\begin{itemize}
 \item The $Ad_K$-invariant $(1,1)$-tensor $\phi$ determined by the endomorphism 
\begin{equation}\label{eq:struc 1}
 ad_{Z}|_{\mathfrak{m}}:{\mathfrak{m}} \rightarrow {\mathfrak{m}},
\end{equation}
where $\mathfrak{m}:=\mathfrak{n}\oplus \R \tilde{C} \cong T_{eK}G/K$ and
$\tilde{C}$ is a complement of $C\in\mathfrak{z(h)}$ if $\dim \mathfrak{z(h)}=2$, and $\tilde{C}=Z$ if $\dim \mathfrak{z(h)}=1$. In case $\dim \mathfrak{z(h)}=2$, we will see that we need to choose $C,\tilde{C}$ generating $\mathfrak{z(h)}$, such that
\[
 B(A,C)=0, \; B(A,\tilde{C})\neq 0.
\]
 \item The vector field $\xi$ determined by $G$-invariance from
 \begin{equation}\label{eq:struc 4}
  \alpha \tilde{C}\in \mathfrak{m},
 \end{equation}
with $\alpha\in \R$, which will be determined in \eqref{alpha}, \eqref{alpha2} by the the compatibility condition \eqref{eq:paracontactmatric3}.
 \item The $G$-invariant $1$-form $\eta$ associated to the $Ad_K$-invariant linear map 
 \begin{equation}\label{eq:struc 2}
  \eta_o: \mathfrak{m} \rightarrow \R; \; \eta_o(\xi)=1, \; \eta_o(X)=0 \text{ for every } X\in \mathfrak{n}.
 \end{equation}
 \item The $G$-invariant metric $g$ on $G/K$ associated to the $Ad_K$-invariant inner product $\langle\langle\;,\;\rangle\rangle$ on $\mathfrak{m}$ defined by  
  \begin{equation}\label{eq:struc 3}
   \langle\langle \xi, \xi\rangle\rangle=1,\; \langle\langle X,Y \rangle\rangle=\langle X,Y\rangle, \; \langle\langle\xi,X\rangle\rangle=0 ,
  \end{equation}
 for every $X,Y\in \mathfrak{n}$.
\end{itemize}
To see that $(G/K,\varphi,\xi,\eta,g)$ is a para-Sasakian $\phi$-symmetric space we use a similar argument as in the Sasakian setting, cf. \cite[Section~4]{KowalskiJimenez}. 

We first show that $(G/K,\varphi,\xi,\eta,g)$ is para-contact. Observe that the properties \eqref{eq:paracontactmatric1},  \eqref{eq:paracontactmatric2} in the definition of para-contact metric manifolds follow directly by construction of $\varphi,\xi,\eta,g$. Thus, we only have to check that \eqref{eq:paracontactmatric3} holds, namely that
 \[
  d\eta_o({X},{Y})=g_o(X,\phi Y), \text{ for every } X,Y\in T_o(G/K)=\mathfrak{m}.
 \]
 Let $X,Y\in \mathfrak{m}$ and denote by  $\tilde{X}$, $\tilde{Y}$ the fundamental vector fields of $X,Y$. 
 As $\eta$ is $G$-invariant we have that
 \begin{equation}\label{etaA}
  0=\mathcal{L}_{\tilde{X}}\eta \tilde{Y}=\tilde{X}(\eta(\tilde{Y}))- \eta[\tilde{X},\tilde{Y}],
 \end{equation}
and hence
 \begin{equation*}
  \begin{aligned}
  2d\eta(\tilde{X},\tilde{Y})&=\tilde{X}(\eta(\tilde{Y}))-\tilde{Y}(\eta(\tilde{X}))-\eta[\tilde{X},\tilde{Y}]\\
  &=-\eta[\tilde{Y},\tilde{X}]=-\eta\widetilde{[{X},{Y}]},
  \end{aligned}
 \end{equation*}
and
\[
 2 d\eta_o({X},{Y})=-\eta_o([{X},{Y}]_{\mathfrak{m}}),
\]
where $[{X},{Y}]_{\mathfrak{m}}$ denotes the projection of $[X,Y]$ on $\mathfrak{m}$. Then, condition \eqref{eq:paracontactmatric3} is equivalent to
\begin{equation}\label{eq:2lambdaB}
 2 \langle\langle X,\phi Y \rangle\rangle=-\eta_o([{X},{Y}]_{\mathfrak{m}}),
\end{equation}
for every $X,Y\in \mathfrak{m}$. Using \eqref{etaA} we have that  equation \eqref{eq:2lambdaB} holds trivially if $X\in \mathbb{R}\tilde{C}$ or $Y\in \mathbb{R}\tilde{C}$. Now we consider $X,Y\in \mathfrak{n}$.
We have that $[X,Y]\in \mathfrak{h}$ and hence
\[
 [X,Y]=\begin{cases}
        a Z + [X,Y]_{[\mathfrak{h},\mathfrak{h}]}, \, \text{if } \dim(\mathfrak{z(h)})=1\\
        a \tilde{C} + b C + [X,Y]_{[\mathfrak{h},\mathfrak{h}]}, \, \text{if } \dim(\mathfrak{z(h)})=2
       \end{cases},
\]
with $a,b\in \R$ and $[X,Y]_{[\mathfrak{h},\mathfrak{h}]}$ the projection of $[X,Y]$ on the commutator ideal ${[\mathfrak{h},\mathfrak{h}]}$. Then, 
\begin{equation}\label{eq:alphaa}
 \eta_o([{X},{Y}]_{\mathfrak{m}})=\alpha^{-1}a.
\end{equation}
On the other hand, as $\mathfrak{h}$ is the centralizer in $\mathfrak{g}$ of $A$, we have that
\[
 B(A,[\mathfrak{h},\mathfrak{h}])=0,
\]
and hence 
\begin{equation*}
 B(A,[X,Y])=\begin{cases}
             B(A,aZ), \, \text{if } \dim(\mathfrak{z(h)})=1\\
             B(A,a \tilde{C} + b C), \, \text{if } \dim(\mathfrak{z(h)})=2
            \end{cases}.
\end{equation*}
Then,
\begin{itemize}
 \item if $\dim(\mathfrak{z(h)})=1$:
 \begin{equation} \label{eq:z1}
 \begin{aligned}
 \langle\langle X,\phi Y \rangle\rangle&=\langle X,ad_Z Y \rangle = \lambda B(X,ad_ZY)\\
 &=-\lambda B([X,Y],Z)=-\lambda B(aZ,Z)\\
 &=- \lambda a \dim N.
 \end{aligned}
\end{equation}
\item if $\dim(\mathfrak{z(h)})=2$:
\begin{equation} \label{eq:z2}
 \begin{aligned}
 \langle\langle X,\phi Y \rangle\rangle&=\langle X,ad_Z Y \rangle=-B([X,Y],A)\\
 &=-a B(\tilde{C},A)-bB(C,A).
 \end{aligned}
\end{equation}
\end{itemize}
If we take 
\begin{itemize}
 \item \begin{equation}\label{alpha}
        \alpha=\frac{1}{2\lambda\dim N},
       \end{equation}
in case $\dim(\mathfrak{z(h)})=1$,
\item $C,\tilde{C}$ generating $\mathfrak{z(h)}$, such that $B(A,C)=0$, $B(A,\tilde{C})\neq 0$ and 
\begin{equation}\label{alpha2}
 \alpha=\frac{1}{2B(A,\tilde{C})},
\end{equation}
in case $\dim(\mathfrak{z(h)})=2$,
\end{itemize}
then
\eqref{eq:2lambdaB} holds, and we have that
$(G/K,\phi,\xi,\eta,g)$ is a para-contact metric manifold.

We prove now that $(\phi,\xi,\eta,g)$ is para-Sasakian, namely that the Nijenhuis tensor $N_\phi$ of $\phi$ vanishes.
Notice that the same argument in \cite[Lemma 4.2.1]{KowalskiJimenez} applies also in the para-contact setting, yielding that $\xi$ is a complete Killing vector field. Then, by \cite[Proposition~2.3]{zamkovoy}
\[
 \mathcal{L}_\xi \phi(X)=0,
\]
for every $X\in T(G/K)$. Observe that for every $X\in T(G/K)$,
\begin{equation*}
 \begin{aligned}
  N_{\phi}(\xi,X)-2d\eta(\xi,X)\xi&=N_{\phi}(\xi , X)\\
  &=\phi^2[\xi,X]-\phi[\xi,\phi X]\\
  &=-\phi (\mathcal{L}_\xi \phi(X))= 0.
 \end{aligned}
\end{equation*}
Let $X,Y$ be  horizontal vector fields. Then, they coincide with the lift of their projection on $G/K$ via
\[
 \pi:G/K \rightarrow G/H,
\]
namely 
\[
 X= (\pi_*X)^*, \; Y=(\pi_*Y)^*.
\]
Moreover, by the definition of $\phi$
\[
 \phi(W)=(I(\pi_*W))^*,
\]
and we have that
\begin{equation*}
 \begin{aligned}
  N_{\phi}(X,Y)-2d\eta(X,Y)\xi=& (I^2(\pi_*[X,Y]))^*+[(I(\pi_*X))^* , (I(\pi_*Y))^*]\\
  &- (I(\pi_*[(I\pi_*X)^* , Y]))^*- (I(\pi_*[X,(I\pi_*Y)^*]))^*\\
  &-2d\eta(X,Y)\xi\\
  =& (I^2([\pi_*X,\pi_*Y]))^*+[I(\pi_*X) , I(\pi_*Y)]^*\\
  &+\eta([\phi X , \phi Y])\xi - \left(I([\pi_*(I\pi_*X\right)^* , \pi_*Y]))^*\\
  &-(I([\pi_*X,\pi_*(I\pi_*Y)^*]))^*-2d\eta(X,Y)\xi\\
  =& ([I,I](\pi_*X,\pi_*Y))^*-2(d\eta(\phi X, \phi Y)+d\eta(X,Y))\xi\\
  =& -2\Big(\omega(I(\pi_* X), I(\pi_* Y))\circ\phi+\omega(\pi_*X,\pi_*Y)\circ \phi\Big)\xi\\
  =&0.
 \end{aligned}
\end{equation*}
Finally we show that $G/K$ is $\phi$-symmetric,
constructing explicitly the $\phi$-geodesic symmetries.
Let $x=gK$ be any element of $G/K$. Observe that the map
\[
 s_x:=g\circ \sigma \circ {g^{-1}}:G/K\rightarrow G/K; aK\mapsto g\sigma(g^{-1}a)K,
\]
is a well defined diffeomorphism fixing the point $x$ and reversing the horizontal geodesics.
In order to show that $s_x$ is an automorphism of the para-Sasakian structure it is sufficient to show that this holds for $x=eK=:o$, with $e$ the identity element of $G$.
Let  $aK\in G/K$ and $v$ be any vector in $T_{aK}G/K$. We consider $X\in \mathfrak{m}$ such that 
\[
 v=\left.\frac{d}{dt}\right|_{t=0}(a\exp(tX) K)
\]
Observe that 
the differential of $s_o$ in $v$ is 
\begin{equation*}
 \begin{aligned}
  (ds_o)_{aK}v &=\left.\frac{d}{dt}\right|_{t=0}s_o (a\exp(tX) K)\\
  &=\left.\frac{d}{dt}\right|_{t=0}\sigma (a) \sigma(\exp(tX)) K\\
  &=(d\sigma (a))_e\circ (d\sigma)_e (X).
 \end{aligned}
\end{equation*}
 Then, since the para-Sasakian structure $(\phi,\xi,\eta,g)$ is induced by $G$-invariance from the $Ad_K$-invariant structure on $\mathfrak{m}$, defined in \eqref{eq:struc 1}, \eqref{eq:struc 4}, \eqref{eq:struc 2}, \eqref{eq:struc 3}, which is preserved by $\sigma$,  we have that $s_o$ preserves $(\phi,\xi,\eta,g)$.
\end{proof}

\begin{theorem}\label{th:semisimple}
 Let $(N,I,g)$ be a semisimple, para-Hermitian symmetric space. Then, $N$ is a product of simple, para-Hermitian symmetric spaces 
 \[
  N=N_1\times \dots \times N_n.
 \]
Moreover, there exists a para-Sasakian $\phi$-symmetric space, fibering via the Boothby--Wang fibration over $N$, if and only if the fundamental $2$-forms $\omega_j$ of $N_j$ are all multiple with the same factor of an integer class, namely there exists $r\in \mathbb{R}$, such that
\[
 [r\,\omega_j]\in H^2(N_j;\mathbb{Z}),
\] 
for every $j\in \{1,\dots,n\}$.
\end{theorem}
\begin{proof}
Let $(\mathfrak{g}, \mathfrak{h}, \sigma, I, g)$ be the effective, semisimple, para-Hermitian symmetric system associated to $(N,I,g)$ and $\mathfrak{g}=\mathfrak{h}\oplus\mathfrak{n}$ be the $\sigma$-decomposition of $\mathfrak{g}$. We have that 
$\mathfrak{g}$ decomposes as sum of simple ideals $\mathfrak{g}_i$
$$\mathfrak{g}=\sum_{i=1}^n \mathfrak{g}_i,$$  
 which are stable under $\sigma$ and together with the restriction of $\sigma$
 $$\sigma_j:=\sigma|_{\mathfrak{g}_j},$$
 give simple symmetric systems $(\mathfrak{g}_j, \mathfrak{h}_j,\sigma_j)$, see \cite[Proposition~3.2]{KaneyukiKozai2}. 
By \cite[Lemma~3.1]{KaneyukiKozai2}, we have that there exists a unique element $Z=Z_1+\dots+Z_n \in \mathfrak{g}$ such that $I=ad_Z$ and $c_{\mathfrak{g}}(Z)=\mathfrak{h}$, with $Z_j\in \mathfrak{g}_j$ and $c_{\mathfrak{g}_j}(Z_j)=\mathfrak{h}_j$. In particular, the para-complex structure 
$I=ad_Z$,
restricts to each simple symmetric system and it is given by
\[
 I_j:=ad_{Z_j}.
\]  
Since the first part of the proof of Proposition~\ref{prop:metric}, and in particular equations \eqref{eq:B-omega} and \eqref{eq:CA}
hold for semisimple para-Hermitian symmetric spaces, we have that 
\[
 g(X,Y)=-B(A,[X,ad_Z Y])
\]
for every $X,Y\in \mathfrak{n}=T_o(G/H)$, $o=eH$,
where $A=A_1+\dots +A_n\in \sum_i\mathfrak{h}_i=\mathfrak{h}$ with $c_{\mathfrak{g}_j}(A_j)=\mathfrak{h}_j$. Then 
\[
 g(X,Y)=-\sum_{j=1}^n B_j(A_j,[X_j,ad_{Z_j}Y_j])
\]
where $X=X_1+\dots+X_n$, $Y=Y_1+\dots+Y_n$, and $X_j, Y_j\in \mathfrak{g}_j$. Then the semisimple, para-Hermitian symmetric system $(\mathfrak{g}, \mathfrak{h}, I, g)$ decomposes as sum of simple, para-Hermitian symmetric system $(\mathfrak{g}_j, \mathfrak{h}_j, ad_{Z_j}, g_j)$, where 
$$g_j(\cdot,\cdot):=-B_j(A_j,[\,\cdot,ad_{Z_j}\cdot])$$  
is either a real multiple of the Killing form $B_j$ of $\mathfrak{g_j}$, if $\mathfrak{h}_j$ as $1$-dimensional center, or it is of the form 
$$g_j(X,Y)= \lambda_j B(\cdot,\cdot)+\mu_j B(i\, \cdot,\cdot), \quad \lambda_j,\mu_j\in \mathbb{R},$$
if $\mathfrak{h}_j$ as $2$-dimensional center.
By Theorem~\ref{th:lambda}, the simple para-Hermitian symmetric spaces $$G_j/H_j=:N_j$$ associated to $(\mathfrak{g}_j, \mathfrak{h}_j, ad_{Z_j}, g_j)$ are base spaces of an $S^1$- or $\mathbb{R}$-bundle 
\[
 \pi_j:M_j\rightarrow N_j.
\]
In particular, the fundamental $2$-form $\omega_j$ of each $N_j$ is a real multiple of an integer class, namely $$[r_j\omega_j]\in H^2(N_j;\mathbb{Z}),$$
with $r_j\in \mathbb{R}$. 
In order to construct an $S^1$- or $\mathbb{R}$-bundle over $G/H$ we recall that it is necessary to assume that the fundamental $2$-form of $G/H$
$$\omega=\omega_1+\dots +\omega_n,$$ 
determines --up to a real constant $r$-- an integer class, namely that
\[
[r\,\omega]\in H^2(N;\mathbb{Z}),
\]
and hence $[r\,\omega_j]\in H^2(N_j;\mathbb{Z})$,
for every $j\in \{1,\dots,n\}$. We consider the $((S^1)^p\times \mathbb{R}^q)$-bundle
\[
 \pi:M_1\times \dots \times M_n\rightarrow N_1\times \dots \times N_n,
\]
with $p,q\in\mathbb{N}$, $p+q=n$, determined by the bundles $\pi_j:M_j\rightarrow N_j$, and observe that
\begin{enumerate}[a)]
\item
If $q=0$, by Section~9 of \cite{Kobayashi}, we know that all $S^1$-bundles over
$N_1\times \dots \times N_n$ are give by the quotient 
of
\[
 M_1\times \dots \times M_n
\]
by an antidiagonal action of a codimension-one torus $T\subset (S^1)^n$
\[
(M_1\times \dots \times M_n)/T\rightarrow N_1\times \dots \times N_n.
\]
\item
Similarly, if $p=0$, an antidiagonal action of $\mathbb{R}^{n-1}\subset \mathbb{R}^n$ on $M_1\times \dots \times M_n$ gives an $\mathbb{R}$-bundle
\[
 (M_1\times \dots \times M_n)/\mathbb{R}^{n-1}\rightarrow N_1\times \dots \times N_n.
\]
\item
If $p,q$ are both nonzero, using a) and b), we have an $(S^1\times \mathbb{R})$-bundle $X$ over
$N_1\times \dots \times N_n$. An antidiagonal action of $\mathbb{R}$ on $X$ gives than an
$S^1$-bundle over $N_1\times \dots \times N_n$
\[
 X/\mathbb{R}\rightarrow N_1\times \dots \times N_n.
\]
\end{enumerate}
In any case, we denote by $M$ the total space of the $S^1$- or $\mathbb{R}$-bundle over $N$ 
\[
 M\rightarrow N,
\]
and observe that, as the para-Sasakian $\phi$-symmetric structure on each $M_j$ and hence on the product $M_1\times \dots \times M_n$ descends to the quotient space $M$, we have that $M$ admits a para-Sasakian $\phi$-symmetric structure.
\end{proof}

\begin{example}
 We consider the symmetric triple (see Section 3 of \cite{KaneyukiKozai2})
 \[
  (sl(p+q,\mathbb{F}), \mathfrak{h} ,\sigma)
 \]
where
\begin{equation*}
 \begin{aligned}
  \mathfrak{h}=\left\{ \begin{pmatrix}
                                            A & 0\\
                                            0 & B
                                           \end{pmatrix}; \; A\in gl(p,\mathbb{F}), B\in gl(q,\mathbb{F}) \text{ and } tr A + tr B= 0\right\},\\
 \end{aligned}
\end{equation*}
\[
 \sigma: sl(p+q,\mathbb{F}) \rightarrow sl(p+q,\mathbb{F}); \; X\mapsto 
                                          \begin{pmatrix}
                                            Id_p & 0\\
                                            0 & -Id_q
                                           \end{pmatrix} X  \begin{pmatrix}
                                                             Id_p & 0\\
                                                             0 & -Id_q
                                                            \end{pmatrix},
\] 
and $\mathbb{F}=\mathbb{R}$ or $\mathbb{C}$.
The $\sigma$-eigenspace decomposition of $sl(p+q,\mathbb{F})$ is given by
\[
 sl(p+q,\mathbb{F})=\mathfrak{h}\oplus \mathfrak{n},
\]
with
\[
       \mathfrak{n}:=\left\{ \begin{pmatrix}
                              0 &A\\
                              B& 0
                             \end{pmatrix};\; A\in M(p,q;\mathbb{F}), B\in M(q,p;\mathbb{F})\right\}.
      \]     
If we take
\[
 Z:=\begin{pmatrix}
     -\frac{q}{p+q}Id_p &0\\
     0&\frac{p}{p+q}Id_q
    \end{pmatrix}\in \mathfrak{h}
\]
we have that $c(Z)=\mathfrak{h}$, and 
\[
 I:=ad_Z|_{\mathfrak{n}}:  \begin{pmatrix}
                              0 &A\\
                              B& 0
                             \end{pmatrix} \mapsto 
                             \begin{pmatrix}
                              0 &-A\\
                              B& 0
                             \end{pmatrix}
\]
is a paracomplex structure on $\mathfrak{n}$. Observe that $(I,\langle \,,\, \rangle)$, where $\langle \,,\, \rangle$ is the restriction to $\mathfrak{n}$ of the Killing form $B$ of $sl(p+q,\mathbb{F})$
\[
 \langle X,Y\rangle:=B|_{\mathfrak{n}}(X,Y),
\]
satisfies the conditions \ref{I}-\ref{IV}, p. \pageref{pagref}, and thus $(sl(p+q,\mathbb{F}),\mathfrak{h},\sigma, I,\langle \,,\, \rangle)$ is a simple para-Hermitian symmetric system.
\begin{itemize}
 \item If $\mathbb{F}=\mathbb{R}$, we have that the center of $\mathfrak{h}$ is $1$-dimensional and generated by $Z$. We consider 
\[
 \mathfrak{k}:=[\mathfrak{h},\mathfrak{h}]=\left\{ \begin{pmatrix}
                              A &0\\
                              0& B
                             \end{pmatrix};\; A\in sl(p,\R), B\in sl(q,\R)\right\}.
                             \]
\item If $\mathbb{F}=\mathbb{C}$, then $\mathfrak{z(h)}=span_{\mathbb{R}}\{Z,iZ\}$ and we take 
\[
 \mathfrak{k}:=[\mathfrak{h},\mathfrak{h}]+span_{\mathbb{R}}\{iZ\}
\]
\end{itemize}
We have that $$sl(p+q,\R)= \mathfrak{m}\oplus \mathfrak{k},$$
with $\mathfrak{m}:=\mathfrak{n}\oplus \R Z$. Let $K$ denote the closed Lie subgroup associated to $\mathfrak{k}$. By Theorem~\ref{th:lambda}, $SL(p+q,\R)/K$ is a para-Sasakian $\phi$-symmetric manifold fibering via the Boothby--Wang fibration, over the para-Hermitian symmetric space $SL(p+q,\R)/H$.
\end{example}

\section{A non-semisimple example}\label{sec:non-semisimple}
We consider the Lie algebra $\mathfrak{sl}(2,\R)$ and the involution $\sigma$ on $\mathfrak{sl}(2,\R)$ defined by
$$\sigma=\text{diag}(1,-1,-1)$$ 
with respect to a basis
 $\{e_1,e_2,e_3\}$ with
 \[
  [e_1,e_2]_{\mathfrak{sl}(2,\R)}=e_3, \; [e_1,e_3]_{\mathfrak{sl}(2,\R)}=e_2, \; [e_2,e_3]_{\mathfrak{sl}(2,\R)}=e_1.
 \]
We denote by $\{\eta_1,\eta_2,\eta_3\} \subset \mathfrak{sl}^*(2,\R)$ and by $\sigma^*$ respectively the dual basis of $\{e_1,e_2,e_3\}$ and the dual map of $\sigma$. 
On the vector space 
$$\mathfrak{d}:=\mathfrak{sl}(2,\R)\oplus \mathfrak{sl}^*(2,\R)$$
we define an involution $\theta$, an inner product $\langle\,,\,\rangle$, and a Lie algebra structure $[\cdot,\cdot]$ as follows: 
\begin{itemize}
 \item  $\theta(\alpha+x):=\sigma^*(\alpha)+\sigma(x), \; \alpha\in \mathfrak{sl}^*(2,\R), x\in \mathfrak{sl}(2,\R)$,
 \item $\langle \alpha_1+x_1,\alpha_2+x_2\rangle:= \alpha_1(x_2)+\alpha_2(x_1), \; \alpha_i\in \mathfrak{sl}^*(2,\R), x_i\in \mathfrak{sl}(2,\R)$,
 \item $ [\cdot,\cdot]:\mathfrak{d}\times\mathfrak{d}\rightarrow \mathfrak{d}$ is the
bilinear, antisymmetric map determined by
\begin{equation*}
 \begin{aligned}
 & \left[e_1,e_2\right]= e_3+c\eta_3, \quad \left[e_1,e_3\right]= e_2-c\eta_2, \quad \left[e_2,e_3\right]= e_1+c\eta_1,\\
 &\left[\eta_1,e_2\right]= \eta_3, \quad \left[\eta_1,e_3\right]= -\eta_2, \quad 
  \left[\eta_2,e_1\right]= \eta_3, \quad  \left[\eta_2,e_3\right]= -\eta_1, \\
 & \left[\eta_3,e_1\right]= \eta_2, \quad  \left[\eta_3,e_2\right]= -\eta_1,
 \end{aligned}
\end{equation*} 
$c\in \mathbb{R}$, $c\neq 0$.
\end{itemize}

We remark that this construction corresponds to a 
quadratic extension of $(\mathfrak{sl}(2,\R), \sigma)$, see \cite[Section~4]{ines}.

\vspace{0.2 cm}

Let $\mathfrak{n}, \mathfrak{h}\subset\mathfrak{d}$ be respectively the $-1$- and $1$-eigenspace of the involution $\theta$,
\[
 \mathfrak{h}=\{e_1,\eta_1\}, \; \mathfrak{n}=\{e_2,e_3,\eta_2,\eta_3\}.
\]
Observe that the endomorphism 
$$J:=ad_{Z}:\mathfrak{n}\rightarrow \mathfrak{n},$$
with $Z=e_1-c\eta_1$, satisfies 
$$J^2=id,$$
and, as 
$\langle \,,\,\rangle$ is $ad_\mathfrak{d}$-invariant, 
$$
\langle J X,Y\rangle+\langle X,JY\rangle=0,  
$$
for every $X,Y\in \mathfrak{n}$.
Thus, by \cite[Proposition~2.4]{KaneyukiKozai2}, the symmetric triple $(\mathfrak{d},\mathfrak{h},\theta)$
is a (non-semisimple) para-Hermitian symmetric system; we denote by $(D/H,g,J)$ the associated (non-semisimple) para-Hermitian symmetric space.

Let $\mathfrak{k}=\mathbb{R}C$ with $C$ any element in $\mathfrak{h}=span_\R\left\{e_1,\eta_1\right\}$. We consider $\widetilde{C}$ any complement of $C$ in $\mathfrak{h}$ and denote 
 $\mathfrak{m}:=\mathfrak{n}\oplus \R\widetilde{C}$.
Observe that the connected Lie subgroup $K$ of $H$ associated to $\mathfrak{k}$ is closed: 
The Lie algebra $\mathfrak{h}$ is $2$-dimensional and abelian, and then the Lie group $H$ is isomorphic to either $\R^2$ or $S^1\times \R$ or $S^1\times S^1$. Notice that $H$ can not be isomorphic to $S^1\times S^1$, otherwise $D/H$ would be a Riemannian symmetric space and the Lie algebra $\mathfrak{d}$ would decompose as direct sum of an abelian and a semisimple Lie algebra (see for instance \cite[Chapter V]{Helgason}).
Looking at the explicit Lie algebra structure of $\mathfrak{d}$ we see that this can not be true, as $\mathfrak{sl}^*(2,\R)$ is an abelian ideal of $\mathfrak{d}$ which can not be contained in the abelian part of $\mathfrak{d}$. Thus, since any Lie subgroup of $\R^2$ and of $S^1\times \R$ is closed, we have that $K$ is closed and $D/K$ is a manifold.
As in the previous section, we consider on $D/K$ the tensors $\xi,\phi,\eta,g$ determined by $D$-invariance, respectively from 
\begin{itemize}
 \item $\xi_0:=\alpha \widetilde{C}\in\mathfrak{m}$, $\alpha\in \mathbb{R}$, $\alpha\neq 0$,
 \item $ ad_Z:\mathfrak{m}\rightarrow \mathfrak{m}$,
 \item $\eta_0:\mathfrak{m}\rightarrow \R \text{ linear, such that } \eta_0(\xi_0)=1, \eta_0(X)=0 \text{ for every } X\in \mathfrak{n}$,
 \item $\langle\langle\;,\;\rangle\rangle: \mathfrak{m}\times \mathfrak{m}\rightarrow \R$ bilinear, symmetric, such that  
 \[
  \langle\langle \xi_0, \xi_0\rangle\rangle=1,\; \langle\langle X,Y \rangle\rangle=\langle X,Y\rangle, \; \langle\langle\xi_0,X\rangle\rangle=0 ,
 \]
 for every $X,Y\in \mathfrak{n}$.
\end{itemize}
These tensors $(\xi,\phi,\eta,g)$ satisfy clearly conditions \eqref{eq:paracontactmatric1}, \eqref{eq:paracontactmatric2}. To prove the compatibility condition \eqref{eq:paracontactmatric3} we consider (instead of the Killing form, considered in the semisimple case) the $ad_\mathfrak{d}$-invariant metric $\langle\,,\rangle$. As in Theorem~\ref{th:lambda}, we have that
\begin{equation*}
 \begin{aligned}
  2 d\eta_0(X,Y)&=-\eta_0([X,Y]_{\mathfrak{m}})
 \end{aligned}
\end{equation*}
where $X,Y\in \mathfrak{m}$, and equation \eqref{eq:paracontactmatric3} is equivalent to
\begin{equation}\label{eq:2lambdaBns}
 2 \langle\langle X,\phi Y \rangle\rangle=-\eta_o([{X},{Y}]_{\mathfrak{m}}),
\end{equation}
for every $X,Y\in \mathfrak{m}$. If $X\in \mathbb{R}\widetilde{C}$ or $Y\in \mathbb{R}\widetilde{C}$, by the definition of Lie bracket on $\mathfrak{d}$, we have that \eqref{eq:2lambdaBns} is trivially satisfied. If $X,Y\in \mathfrak{n}$, then $[X,Y]\in \mathfrak{h}$,
\[
 [X,Y]=a C + b \widetilde{C}
\]
with $a,b\in \R$. Thus, 
\[
 \eta_o([{X},{Y}]_{\mathfrak{m}})=\eta_o(b \widetilde{C})=\alpha^{-1}b,
\]
and
\begin{equation*}
 \begin{aligned}
  \langle\langle X,\phi Y\rangle\rangle &= \langle X,ad_Z Y\rangle=-\langle [X,Y],Z \rangle=-a\langle C,Z \rangle-b\langle \widetilde{C},Z\rangle,
 \end{aligned}
\end{equation*}
and then \eqref{eq:2lambdaBns} holds if and if and only if 
$$\langle C,Z \rangle=0, \; \langle \widetilde{C},Z\rangle\neq 0, \; \alpha= -\frac{1}{2\langle\widetilde{C},Z\rangle}.$$ 
As in the last part of Theorem~\ref{th:lambda}, we have that the para-contact metric structure $(\xi,\phi,\eta,g)$ on $D/K$ is para-Sasakian $\phi$-symmetric.

\end{document}